\documentclass[12pt,reqno]{article}
\usepackage{amsmath,amssymb,amsfonts,amscd,latexsym,amsthm,mathrsfs,verbatim,comment}
\usepackage[usenames]{color}
\usepackage{enumerate}
\usepackage{float}
\usepackage[unicode]{hyperref}
\newtheorem{theorem}{Theorem}[section]
\newtheorem{proposition}[theorem]{Proposition}
\newtheorem{corollary}[theorem]{Corollary}
\newtheorem{lemma}[theorem]{Lemma}
\theoremstyle{remark}
\newtheorem{definition}[theorem]{\bf Definition}
\newtheorem{remark}[theorem]{\bf Remark}
\newtheorem{remarks}[theorem]{\bf Remarks}

\DeclareMathOperator{\atanh}{atanh}

\DeclareMathOperator{\PSL}{PSL}
\DeclareMathOperator{\CS}{CS}

\newcommand{\Z}{\mathbb Z}
\newcommand{\Q}{\mathbb Q}
\newcommand{\G}{\Gamma}

\newcommand{\z}{\zeta}

\newcommand{\Om}{\Omega}
\newcommand{\lcm}{\operatorname{lcm}}

\newcommand{\cP}{\mathcal{P}}
\renewcommand\pmod[1]{\;(\operatorname{mod}#1)}

\begin{document}
\pagestyle{plain}

\title{Continued Fractions and Irrationality Measures for Chowla--Selberg Gamma Quotients}

\author{Henri Cohen (Bordeaux) and Wadim Zudilin (Nijmegen)}

%\date{11 July 2026}

\maketitle

\begin{abstract}
  We give 39 rapidly convergent continued fractions
  for Chowla--Selberg gamma quotients, and deduce good irrationality measures
  for 20 of them, including for $\operatorname{CS}(-3)=(\Gamma(1/3)/\Gamma(2/3))^3$,
  for $a^{1/4}\operatorname{CS}(-4)=a^{1/4}(\Gamma(1/4)/\Gamma(3/4))^2$ with $a=12$ and $a=1/5$, and for
  $\operatorname{CS}(-7)=\Gamma(1/7)\Gamma(2/7)\Gamma(4/7)/(\Gamma(3/7)\Gamma(5/7)\Gamma(6/7))$.
  These appear to be the first proved and reasonable irrationality measures for gamma quotients.
\end{abstract}  
  
\section{Introduction}
\label{sec:intro}
Let $D$ be a negative fundamental discriminant, let $\delta=0$ or $1$
such that $D\equiv\delta\pmod4$, and denote by $w(D)$ and $h(D)$ the number
of roots of unity and the class number of $\Q(\sqrt{D})$.

\begin{definition} We define the \emph{Chowla--Selberg gamma quotient} by
  $$\CS(D)=\Bigg(\prod_{j=1}^{|D|}\G\bigg(\frac j{|D|}\bigg)^{\left(\frac{D}{j}\right)}\Bigg)^{w(D)/(2h(D))}\;.$$
\end{definition}

The importance of these expressions comes from the Lerch, Chowla--Selberg
formula and generalizations, which connects the value of the Dedekind
eta function at CM points of discriminant $D$ with $\CS(D)$. For instance,
if $h(D)=1$ we have $|\eta((-\delta+\sqrt{D})/2)|^4=\CS(D)/2\pi|D|$.

It is known since Chudnovsky and Nesterenko \cite{Chu, Wal, Nes} that
$\CS(D)$, $\pi$, and $\exp(\pi\sqrt{|D|})$ are algebraically independent over
$\Q$, and in particular that $\CS(D)$ is transcendental. None\-theless,
explicit and \emph{reasonable} irrationality measures\footnote{We recall
that the \emph{irrationality measure} $\mu(L)$ of an irrational real $L$ is defined as the
supremum of the set of real numbers $\mu$ for which
$0<|L-p/q|<1/\max(|p|,|q|)^{\mu}$ has infinitely many solutions
$(p,q)\in\Z\times\Z_{\ne0}$. Then $\mu(L)\ge 2$ and it is exactly $2$ when $L$ is algebraic by the celebrated theorem of Roth.
When $L$ is transcendental, $\mu(\L)$ can assume any value in the range $2\le\mu(L)\le\infty$, with $\mu(L)=\infty$ designating
Liouvillian numbers; $\mu(L)=2$ is nevertheless most typical as it valid for almost every real number $L$ in the sense of Lebesgue measure.
Proving bounds for concrete \emph{interesting} numbers is a difficult problem.}
for these numbers are rare, and we mention
one for $\CS(-3)$ which is \emph{experimentally} deduced in \cite{D-BKZ22}
based on an explicit construction of rational approximations to the internally
defined constant $K(0, 1/3, 2/3, 1/3, 2/3)$ (identified in \cite{Zud21}
as a M\"obius transform of $\CS(-3)$). The recent paper \cite{Zud25} by the second author
takes a very similar approach to the one used here, but for a slightly
different problem. Our goal is thus to give good irrationality measures for
$\CS(D)$ (possibly multiplied by some simple algebraic number) for quite a
number of $D$, including for $D=-3$, $-4$, and $-7$,
see Theorem \ref{thm:irrall} for the complete list.
We believe that these are the first known \emph{proved} bounds for the irrationality measures of quantities linked to gamma quotients, disregarding related achievements in \cite{Zud25} and gargantuan bounds from \cite{Bru02} for $\mu(\G(1/3))$ and $\mu(\G(1/4))$.

\subsection{Obtaining Irrationality Measures}

The usual method for obtaining irrationality measures for some irrational
number $L$ is to construct two sequences of rational numbers $u(n)$ and
$v(n)$ such that $|v(n)L-u(n)|$ is small and the denominators of
$u(n)$ and $v(n)$ are not too large. More precisely, we have the following
elementary result:
\begin{lemma}\label{lem:irrmea}
  Let $L\notin\Q$, and let $u(n)$ and $v(n)$ be two sequences
  of rational numbers such that as $n\to\infty$:
  \begin{enumerate}\item We have $\log|v(n)L-u(n)|\sim-Fn$ and $\log|v(n)|\sim Fn$ for some positive
    constant $F$.
  \item There exists a sequence of positive integers $d(n)$ such that
    $d(n)u(n)\in\Z$ and $d(n)v(n)\in\Z$ with $\log(d(n))<Mn$ for some
    constant $M$.
  \end{enumerate}
  Then if $F>M$, an upper bound on the irrationality measure $\mu$ of $L$
  is given by $1+(F+\nobreak M)/(F-M)$.
\end{lemma}

Most results giving irrationality measures, such as the earliest one for $\pi$
due to Mahler, or the best one due to Zeilberger and the second author,
are obtained by suitable integral and hypergeometric-like constructions of the
sequences $u(n)$ and $v(n)$. It is quite rare that these sequences are
obtained from convergents of continued fractions, with the notable exceptions
of the results of Ap\'ery for $\pi^2$ and $\z(3)$ (of course a measure for
$\pi^2$ implies one for $\pi$, and these Ap\'ery measures have since been
improved using the abovementioned constructions).
Thus, it is quite surprising that we are
going to show that a very simple family of continued fractions (much simpler
than Ap\'ery's) give essentially the first known irrationality measures for
gamma quotients, and as far as we can tell, there is no integral or series
that can even approach this type of result.

\subsection{A Motivating Example}
\label{ssec:1.1}

We could directly delve into the results and the proofs of our results,
but we believe that it is instructive to give a leisurely account of what
led to them, since it also gives additional results and insights.

Recall first a very practical notation for continued fractions (CF), used
for instance in \cite{Coh2} and \cite{Coh3}: An expression of the type
$L=[[a_0,a_1,a(n)],[b_0,b(n)]]$, where $a(n)$ and $b(n)$ are polynomials in
$n$, means that $L$ is the limit of the continued fraction
$$L=a_0+b_0/(a_1+b(1)/(a(2)+b(2)/(a(3)+b(3)/(a(4)+\cdots))))\;.$$

In \cite{Coh3} it was noticed that, due to the $abc$ triple $5^3+3=2^7$ and
to a classical continued fraction due to Laguerre, we can easily
construct a CF for $2^{1/3}$ with a remarkably large speed of convergence.
More precisely:
\begin{equation}
2^{1/3}=[[5/4,252,253(2n-1)],[5/2,-(9n^2-1)]]
\label{CF-cubic2}
\end{equation}
with speed of convergence
$$2^{1/3}-\dfrac{p(n)}{q(n)}\sim\dfrac{2^{4/3}3^{3/2}}{(16+5\sqrt{10})^{4n+2}6^{-2n}}\;.$$
Note that $E=(16+5\sqrt{10})^4/6^2=28446.444\cdots$, and that the study
of the denominators gives an explicit irrationality measure
$\mu(2^{1/3})<2.827$, however not as good as the best known.

\smallskip

Using an idea already exploited for instance in \cite{Gol-Zag}, and that we
used in \cite{Coh-Zud}, we can do a \emph{half-shift} of this CF, in
other words change $n$ into $n-1/2$, then compute \emph{numerically} the limit
of this new CF, and\,---\,thanks to the Encyclopedia described in \cite{Coh2}
and \cite{Coh3}\,---\, to recognize this limit and
deduce the (conjectural) continued fraction
\begin{equation}
\CS(-3)=\left(\dfrac{\G(1/3)}{\G(2/3)}\right)^3=[[0,31,1012(n-1)],[240,-(6n-1)(6n-5)]]
\label{CF-CS3}
\end{equation}
with essentially the same speed of convergence
$$\CS(-3)-\dfrac{p(n)}{q(n)}\sim\dfrac{3^{3/2}\CS(-3)}{(16+5\sqrt{10})^{4n}6^{-2n}}\,.$$

\smallskip

At least three questions now arise: First of course, how do we \emph{prove}
the validity of this CF? Second, even once this is done, is there a deeper
reason for the existence of such a rapidly convergent CF for a gamma quotient?
And third, does this give a good irrationality measure for
$\CS(-3)$?

The purpose of this paper is to answer positively all three questions and, in
particular, to give other examples of similar rapidly convergent CFs for gamma
quotients, and whenever possible to deduce\,---\,in a quantitative
form\,---\,the irrationality of these numbers.

\smallskip

For future reference, note the following easy lemma:

\begin{lemma}\label{lem:conj}
  Denote by $p(n)$ and $q(n)$ the numerators and denominators of the above
  CF, and set $f(n)=\prod_{1\le j\le n}(6j-5)$. Then $v_n=p'(n)=p(n)/f(n)$ and
  $v_n=q'(n)=q(n)/f(n)$ are both solutions of the recursion
  $$(6n+1)v_{n+1}-1012nv_n+(6n-1)v_{n-1}=0\;,$$
  and
  $$\log(|q'(n)\CS(-3)-p'(n)|)\sim-n\log((16+5\sqrt{10})^2/6)\;.$$
\end{lemma}

Since this will always be the case, note in passing that
$(16+5\sqrt{10})^2/6=E^{1/2}$, where $E$ is given above.

\section{Prelude: A Continued Fraction for the Power Function}

Before beginning our study, it is interesting to understand the origin of
the rapidly convergent CF \eqref{CF-cubic2} for~$2^{1/3}$.
Using Euler's transformation of series into CFs, it is trivial to transform
the Taylor expansion of $(1+z)^a$ into the following CF:

\begin{lemma} We have the CF
$$(1+z)^a=[[0,1,(n-1)-(n-2-a)z],[1,-az,(n-1)(n-1-a)z]]$$
with speed of convergence
$$(1+z)^a-\dfrac{p(n)}{q(n)}\sim\dfrac{1/((z+1)\G(-a))}{(-1/z)^nn^{a+1}}\;.$$
\end{lemma}

If we apply this to $a=-1/3$ and $z=-3/128$, so that $(1+z)^a=(4/5)2^{1/3}$,
this gives a CF for $2^{1/3}$ which converges essentially in $(128/3)^{-n}$,
which is already reasonably fast. But the CF mentioned in the previous
section converges much faster, and this is because we implement a better CF for
$(1+z)^a$, using an
idea as old as calculus itself: it is well-known that if you really want
to compute a logarithm using a power series, instead of using the Taylor
expansion of $\log(1+z)$ it is better to use the Taylor expansion of
$\log((1+z)/(1-z))=2\atanh(z)$ which converges much faster, and has the
added advantage that the M\"obius transformation $z\mapsto(1+z)/(1-z)$ is
invertible. The CFs that we want are the following:

\begin{proposition}[Laguerre]\label{prop:lag} We have the CFs
$$\left(\dfrac{1+z}{1-z}\right)^a=[[1,1-az,2n-1],[2az,-z^2(n^2-a^2)]]$$
with speed of convergence
$$\left(\dfrac{1+z}{1-z}\right)^a-\dfrac{p(n)}{q(n)}\sim\dfrac{2\sin(\pi a)((1+z)/(1-z))^a}{((1+\sqrt{1-z^2})/z)^{2n+1}}\;,$$
or equivalently
$$(1+z)^a=[[1,z(1-a)+2,(z+2)(2n-1)],[2az,-z^2(n^2-a^2)]]$$
with speed of convergence
$$(1+z)^a-\dfrac{p(n)}{q(n)}\sim\dfrac{2\sin(\pi a)(1+z)^a}{(1+\sqrt{1+z})^{4n+2}/z^{2n+1}}\;.$$
\end{proposition}

Note that there is apparently no simple expression for the Taylor \emph{series}
expansion of $((1+z)/(1-z))^a$.

To prove this result we need a series of lemmas, all essentially due to Gauss.

\begin{lemma}\label{lem:gauss} We have the following contiguity relations:
  \begin{align*}
    {}_2F_1(a,b;c;z)&={}_2F_1(a,b+1;c+1;z)-\dfrac{a(c-b)}{c(c+1)}z\cdot{}_2F_1(a+1,b+1;c+2,z)\;,\\
    {}_2F_1(a,b;c;z)&={}_2F_1(a+1,b;c+1;z)-\dfrac{b(c-a)}{c(c+1)}z\cdot{}_2F_1(a+1,b+1;c+2,z)\;.
  \end{align*}
\end{lemma}

\begin{proof}
The identities are trivially checked on the power series expansion
of ${}_2F_1(a,b;c;z)$, and are also equivalent by exchanging $a$ and $b$.
\end{proof}

\begin{corollary}
  Fix $a$, $b$, and $c$. We have the continued fraction
    $$\dfrac{{}_2F_1(a,b;c;z)}{{}_2F_1(a,b+1;c+1;z)}=1+a_1z/(1+a_2z/(1+a_3z/(1+a_4z/(1+\cdots))))\;,$$
    with
    $$a_{2n+1}=-\dfrac{(a+n)(c-b+n)}{(c+2n)(c+2n+1)}\text{\quad and\quad}
    a_{2n+2}=-\dfrac{(b+n+1)(c-a+n+1)}{(c+2n+1)(c+2n+2)}\;.$$
\end{corollary}

\begin{proof}
Set
  \begin{align*}
    R_{2n}(z)&=\dfrac{{}_2F_1(a+n,b+n,c+2n;z)}{{}_2F_1(a+n,b+n+1;c+2n+1;z)}\text{\quad and}\\
    R_{2n+1}(z)&=\dfrac{{}_2F_1(a+n,b+n+1;c+2n+1;z)}{{}_2F_1(a+n+1,b+n+1;c+2(n+1);z)}\;.
  \end{align*}
  Applying Lemma \ref{lem:gauss} it is clear that we have the recursion
  $R_n=1+a_{n+1}/R_{n+1}$, where $a_{n+1}$ is given by the formulas in
  the corollary, and the continued fraction follows.
\end{proof}

\begin{corollary}\label{cor:gauss2} We have the continued fraction
  $$\dfrac{{}_2F_1(a,a-1/2;c;z)}{{}_2F_1(a,a+1/2;c+1;z)}=[[1,2(n+c)],[-(a/c)(2(c-a)+1),-z(n+2a)(n+2(c-a)+1)]]\;.$$
\end{corollary}

\begin{proof}
Indeed, it is immediate to check that when $b=a-1/2$, the formulas
for $a_{2n+1}$ and $a_{2n+2}$ coincide, so the CF follows after
simplifying denominators.
\end{proof}

\begin{proof}[Proof of Proposition \ref{prop:lag}]
Expanding by the binomial theorem, we see immediately that
\begin{align*}
  (1+z)^a+(1-z)^a&=2\cdot{}_2F_1((1-a)/2,-a/2;1/2;z^2)\;,\\
  (1+z)^a-(1-z)^a&=2az\cdot{}_2F_1((1-a)/2,(2-a)/2;3/2;z^2)\;.
\end{align*}
We apply Corollary \ref{cor:gauss2} with $(a,b,c,z)$ replaced by
$((1-a)/2,-a/2,1/2,z^2)$, which is applicable since the difference of the
first two parameters is $1/2$, and we find the CF
$$2az\dfrac{(1+z)^a+(1-z)^a}{(1+z)^a-(1-z)^a}=[[2n+1],[-z^2((n+1)^2-a^2)]]\;.$$
If we denote by ${\cal C}$ this last CF we thus have
$((1+z)/(1-z))^a=1+2az/(-az+{\cal C})$, and the first CF of the proposition
follows. Changing $z$ into $z/(z+2)$ and clearing denominators gives the
second CF.
\end{proof}

Choosing $a=-1/3$ and $z=-3/128$ gives the very rapidly convergent CF
\eqref{CF-cubic2} for $2^{1/3}$ from the introduction.

\begin{corollary}\label{cor:2}
  Denote by $p(n)$ and $q(n)$ the numerators and denominators of the
  CF for $(1+z)^a$ in Proposition \ref{prop:lag}, and set
  $f(n)=z^n\prod_{1\le j\le n}(j-a)$. Then $v_n=p'(n)=p(n)/f(n)$ and
  $v_n=q'(n)=q(n)/f(n)$ are both solutions of the recursion
  $$(n+1-a)v_{n+1}-(1+2/z)(2n+1)v_n+(n+a)v_{n-1}=0$$
  with initial values $p'_0=q'_0=1$, $p'_1=2a/(1-a)$, and
  $q'_1=1+2/(z(1-a))$.
\end{corollary}

\begin{proof}
Clear.
\end{proof}

\section{A Family of Continued Fractions}

We now introduce the family of continued fractions (or, equivalently, of
recursions) that we are going to study. We will restrict to CFs of the
following type:
$${\cal C}=[[0,a_1,A(n-1)],[b_0,-K(Dn-1)(D(n-1)+1)]]\;,$$
with $A>0$, $K\ne0$, and in the cases that we are interested in,
$D=2$, $3$, $4$, or $6$.

\noindent
\textbf{Warning}: From now on, we use the letter $D$ for the integral positive parameter appearing in the above continued fraction, and not anymore for negative fundamental discriminants.

\begin{proposition}\label{prop:asymp}
  For $A>0$, $K\ne0$, and $D\ge2$, let ${\cal C}$ be the continued fraction
  $${\cal C}=[[0,a_1,A(n-1)],[b_0,-K(Dn-1)(D(n-1)+1)]]\;.$$
  Assume that $A^2-4KD^2>0$, and define
  $$R=\dfrac{A+\sqrt{A^2-4KD^2}}{2}\quad\text{and}\quad E=\dfrac{R^2}{KD^2}\;.$$
  Denote by $(p(n),q(n))$ the $n$th partial quotients of ${\cal C}$,
  and set $(p'(n),q'(n))=(p(n),q(n))/f(n)$, where
  $$f(n)=|K|^{\lfloor n/2\rfloor}\prod_{1\le j\le n}(D(j-1)+1)\;.$$
  Then ${\cal C}$ converges exponentially fast to some limit $L$, and we
  have the following asymptotics, where $C_1$ and $C_2$ are some nonzero
  constants:
  \begin{align*}
    L-p(n)/q(n)&\sim C_1/E^n\;,\quad q(n)\sim C_2(n-1)!R^n\;,\\
    \log(|q'(n)L-p'(n)|)&\sim-n\log(|E|)/2\;.
  \end{align*}
\end{proposition}

\begin{proof} Classical and left to the reader.\end{proof}

In the next section, we will see up to three methods for computing the
limit $L$ of these continued fractions.

Although we do not need it, note that $\cal C$ can also be written
$${\cal C}=[[b_0/a_1,A-(D-1)K/a_1,An],[(D-1)Kb_0/a_1^2,-K(Dn+1)(Dn+D-1)]]\;.$$

\section{Computing the Limit}

In this section, we will give three different expressions for the limit $L$,
the third valid only for $D=2$.

\smallskip

\subsection{Using Laguerre's CF}

\smallskip

\begin{theorem}\label{thm:alt}
  Keep the assumptions and notation of Proposition \ref{prop:asymp}.
  We have $L=b_0/(a_1+L_0)$ with
  $$L_0=Af_D(4KD^2/A^2)\;,\text{\quad where\quad}  
  f_D(z)=z\,\dfrac{{}_2F_1'(1/(2D),1/(2D)-1/2;1;z)}{{}_2F_1(1/(2D),1/(2D)+1/2;1;z)}\;,$$
  and the contiguity relation
  $$f_D(z)=-\dfrac{z}{2D}+z(1-z)\,\dfrac{{}_2F_1'(1/(2D),1/(2D)+1/2;1;z)}{{}_2F_1(1/(2D),1/(2D)+1/2;1;z)}\;.$$
\end{theorem}

\begin{proof}
  We first multiply the equation of Corollary \ref{cor:gauss2} by $c$ and then
  set $c=0$. Since $\lim_{c\to0}(c)_n/c=n!/n$, we obtain
  $$\dfrac{{}_2F_1'(a,a-1/2;1;z)}{{}_2F_1(a,a+1/2;1;z)}=[[0,2n],[a(2a-1),-z(n+2a)(n-2a+1)]]\;.$$
  Choosing $a=1/(2D)$, $z=4KD^2/A^2$, and clearing denominators proves the
  formula for $L$. As usual, the contiguity relation is trivially checked
  by a direct computation.
\end{proof}

As we will see below, the main point of this type of evaluation is that it
is closely related to the evaluations of modular forms at CM values.
For instance, recall the famous formula due to Fricke
$${}_2F_1(1/12,7/12;1;1728/(1728-j(\tau)))=E_6^{1/6}(\tau)\;,$$
valid for $\tau$ in the usual fundamental domain of $\PSL_2(\Z)$.
The left-hand side is thus the denominator of $f_6(1728/(1728-j(\tau)))$.
When $\tau$ is a CM point, $j(\tau)$ is algebraic, $E_6(\tau)$ is equal
to an algebraic number times a gamma quotient, and similarly for its
derivative used in the contiguity relation, so it is clear that the limit
is related to gamma quotients. We could use this to obtain our results,
but we have preferred to use a second method which we now explain.

\smallskip

\subsection{Another Hypergeometric Expression for the Limit}

\smallskip

\begin{lemma}\label{lem:tnrec}
  Set
  \begin{align*}
    T_n(a,b;z)&=\dfrac{\G(a+n)\G(b+n)}{\G(2a+2n)}(1/z)^{a+n}{}_2F_1(a+n,b+n;2a+2n;1/z)\\
    U_n(a,b;z)&=(-1)^n\cdot{}_2F_1(1-a-n,a+n;1+a-b;z)\;.
  \end{align*}
  Then both $v_n=T_n(a,b;z)$ and $v_n=U_n(a,b;z)$ satisfy the recursion
\begin{equation}
(2a-b+n)v_{n+1}-(2z-1)(2a+2n-1)v_n+(b+n-1)v_{n-1}=0\;.
\label{rec-main}
\end{equation}
\end{lemma}

\begin{proof}
The proof is an elementary exercise on the contiguity relations of
hypergeometric series, or more simply by checking vanishing of the
power series expansion in~$z$. Alternatively, one can apply creative telescoping algorithms.
For completeness we have given the result for $U_n$, but we will only use $T_n$.
\end{proof}

\begin{remarks}
  \begin{enumerate}\item
    The remarkable aspect of this general recursion is that its coefficients
    are only \emph{linear} in $n$, while more general recursions for
    hypergeometric functions would be at least quadratic.
    \item This recursion can be identified with that of Lemma \ref{lem:conj}
      by choosing $a=1/2$, $b=5/6$, and $z=128/3$, and with that of Corollary
      \ref{cor:2} by choosing $a=1$, $b=a+1$ and replacing $z$ by $(z+1)/z$.
    \item Note also that this indeed corresponds to shifting by $1/2$ since
      trivially
\[ T_{n-1/2}(1,a+1,z)=T_n(1/2,a+1/2,z)\;, \]
so for instance in the above cases
$T_{n-1/2}(1,4/3,z)=T_n(1/2,5/6,z)$,
and for the cases below we have
$T_{n-1/2}(1,5/4,z)=T_n(1/2,3/4,z)$ and $T_{n-1/2}(1,7/6,z)=T_n(1/2,2/3,z)$.
  \end{enumerate}
\end{remarks}

\begin{corollary}\label{cor:1}
  Assume that $b\ne2a$. For $z>1$, consider the CF
  $$[[a_0,a_1,(2z-1)(2a+2n-3)],[b_0,-(b+n-1)(2a-b+n-1)]]\;.$$
  It converges generically like $(\sqrt{z}+\sqrt{z-1})^{-4n}$ to a limit $L$
  given by the formula $L=a_0+b_0/(a_1+L_0)$ with
  $$L_0=(b-2a)\dfrac{T_1(a,b;z)}{T_0(a,b;z)}\;.$$
\end{corollary}

\begin{proof}
Denote as usual by $p(n)$ and $q(n)$ the partial quotients of this
CF and set similarly as above $(p'(n),q'(n))=(p(n),q(n))/f(n)$ with
$f(n)=\prod_{1\le j\le n}(2a-b+j-1)$, so that $v_n=p'(n)$ and $q'(n)$
both satisfy the recursion of the lemma. We check immediately that as
$n\to\infty$, $q'(n)$ is asymptotic to $C_1(\sqrt{z}+\sqrt{z-1})^{2n}n^{b-a-1/2}$ and $(p'(n)/q'(n))-L$ is asymptotic to $C_2(\sqrt{z}+\sqrt{z-1})^{-4n}$
for some constants $C_1$ and $C_2$, so $p'(n)-Lq'(n)$ converges
to $0$ exponentially fast, essentially as $(\sqrt{z}+\sqrt{z-1})^{-2n}$.

On the other hand, from the integral representation of the hypergeometric
function it is immediate to check that $T_n(a,b;z)$ also tends to $0$
exponentially fast, in fact also essentially as $(\sqrt{z}+\sqrt{z-1})^{-2n}$.

Since the general solution of the above linear recursion is of the form
$A(p'(n)-Lq'(n))+Bq'(n)$ and $q'(n)$ tends to infinity exponentially
fast, it follows that there exists a constant $A$ such that
$T_n(a,b;z)=A(p'(n)-Lq'(n))$. In particular,
$$\dfrac{T_1(a,b;z)}{T_0(a,b;z)}=\dfrac{p'(1)-Lq'(1)}{p'(0)-Lq'(0)}=\dfrac{a_0a_1+b_0-La_1}{(2a-b)(a_0-L)}\;,$$
proving the result.
\end{proof}

It follows from this corollary that to compute the limit $L$ of the CF, when
it is unknown, it is sufficient to compute ${}_2F_1(a,b;2a;1/z)$ and
${}_2F_1(a+1,b+1;2a+2;1/z)$. For this purpose, note the following:
\begin{lemma}\label{lem:1}
  We have the contiguity relation
  $${}_2F_1(a+1,b+1;2a+2;x)=\dfrac{2(2a+1)}{(2a-b)x}{}_2F_1(a,b;2a;x)
  +\dfrac{4(x-1)(2a+1)}{b(2a-b)x}{}_2F_1'(a,b;2a;x)\;.$$
\end{lemma}

\begin{proof}
Immediate exercise on contiguity relations.
\end{proof}

But conversely, if $L$ is known, one deduces the value of quotients of
hypergeometric functions. For instance:

\begin{proposition}
We have
  $$\dfrac{T_1(1,a+1;z)}{T_0(1,a+1;z)}=\dfrac{a+1}{6z}\dfrac{{}_2F_1(2,a+2;4;1/z)}{{}_2F_1(1,a+1;2;1/z)}=\dfrac{(1+a-2z)-(1-a-2z)(1-1/z)^a}{(a-1)(1-(1-1/z)^a)}\;.$$
\end{proposition}

Note that this can be easily shown directly: for example, it is immediate to
see that $T_0(1,a+1;z)=\G(a)((1-1/z)^{-a}-1)$ and
$T_1(1,a+1;z)=\G(a-1)((a+1-2z)(1-1/z)^{-a}+a-1+2z)$, from which one recovers
the above formula.

\smallskip

The analogue of Theorem \ref{thm:alt} is the following:

\begin{theorem}\label{thm:main2}
  Keep the assumptions and notation of Proposition \ref{prop:asymp}
  and Theorem \ref{thm:alt}. We have $L=b_0/(a_1+L_0)$ with
    $$L_0=-K^{1/2}\dfrac{T_1}{T_0}\left(\dfrac{1}{2},1-\dfrac{1}{D};\dfrac{1}{2}+\dfrac{A}{4DK^{1/2}}\right)\;.$$
\end{theorem}

\begin{proof} Immediate consequence of Corollary \ref{cor:1}. Note that in
  this formula $K$ may be negative, but the right-hand side is always real.
\end{proof}

\smallskip

\subsection{The Limit as a Complete Elliptic Integral}

\smallskip

In the case $D=2$, we can also express the limit in terms of complete
elliptic integrals thanks to the following theorem:

\begin{theorem}\label{thm:EK}
  Keep the assumptions and notation of the previous theorems.
  In the case $D=2$, we have $L=b_0/(a_1+L_0)$ with
  $$L_0=-\dfrac{A}{4}-\dfrac{A/2}{Q((2K^{1/2}/(A/4+K^{1/2}))^{1/2})}\;,\quad\text{where}\quad Q(k)=(k^2-2)\dfrac{K(k)}{E(k)}\;.$$
\end{theorem}

\noindent
\textbf{Warning}: Do not confuse the \emph{function} $K(k)$ with the variable $K$.

\begin{proof}
  From Theorem 7.2 of \cite{Coh-Zud} we have the following continued fraction:
  $$\dfrac{E(k)}{K(k)}=[(2-k^2)[1/2,4n],-k^4[1/2,(2n+1)^2]]\;.$$
  Set $B=4K^{1/2}(2-k^2)/k^2$. Multiplying the $a(n)$ of this CF by
  $K^{1/2}/k^2$ and the $b(n)$ by $K/k^4$ gives an equivalent CF, so
  $$\dfrac{E(k)}{K(k)}=[[1-k^2/2,Bn],[-k^2K^{1/2}/2,-K(2n+1)^2]]\;.$$
  Let $U=B-3^2K/(2B-5^2K/(3B-\cdots))$ be the CF beginning
  at $n=1$. We have by definition $E(k)/K(k)=1-k^2/2-k^2K^{1/2}/(2U)$,
  so $U=k^2K^{1/2}/2/(1-k^2/2-E(k)/K(k))$. In addition, we check that for
  $k^2=2K^{1/2}/(A/4+K^{1/2})$ we have $B=A$.

  On the other hand, we have $L=[[0,a_1,A(n-1)],[b_0,-K(2n-1)^2]]$, so
  $$L=b_0/(a_1-K/(A-3^2K/(2A-5^2/(3A-\cdots))))=b_0/(a_1-K/U)\;,$$
  and the theorem follows after simplifications.
\end{proof}

In the coming sections we will compute explicitly this limit $L$ for
a few values of $A$ and $K$ in terms of Chowla--Selberg quotients.
Thanks to the above theorem, this gives so-called singular value evaluations of
$K(k)/E(k)$ for some values of $k$. Thus, using the results of the next
sections, we obtain the following:

\begin{corollary}
  Recall that $Q(k)=(2-k^2)K(k)/E(k)$. We have the following evaluations:
  \begin{align*}
    Q(1/\sqrt{2})&=\dfrac{3}{1+4/\CS(-4)}\;,\\
    Q\big(2\sqrt{\vphantom|\smash[t]{3\sqrt{2}-4}}\big)&=\dfrac{6}{1+8/\CS(-4)}\;,\\
    Q\big(\sqrt{\vphantom|\smash[t]{2\sqrt{2}-2}}\big)&=\dfrac{4}{1+8/\CS(-8)}\;,\\
    Q\big((\sqrt{3}+\sqrt{-1})/2\big)&=\dfrac{3}{1+6/(2^{1/3}\CS(-3))}\;,\\
    Q\big((3+\sqrt{-7})/8\big)&=\dfrac{21}{8+28/\CS(-7)}\;.
  \end{align*}
\end{corollary}

Note that this comes directly from our results below, but it is easy to
obtain explicit evaluations of $K(k)$ and $E(k)$ separately for the same
values of~$k$:
\begin{align*}
K(1/\sqrt{2})&=2^{-7/4}(\G(1/4)^3/\G(3/4))^{1/2}\;,\\
K\big(2\sqrt{\vphantom|\smash[t]{3\sqrt{2}-4}}\big)&=2^{-9/4}(1+\sqrt{2})(\G(1/4)^3/\G(3/4))^{1/2}\;,\\
K\big(\sqrt{\vphantom|\smash[t]{2\sqrt{2}-2}}\big)&=2^{-11/4}(1+\sqrt{2})^{1/2}\G(1/8)\G(3/8)/\pi^{1/2}\;,\\
K\big((\sqrt{3}+\sqrt{-1})/2\big)&=2^{-4/3}3^{-1/4}e^{2\pi i/24}\G(1/3)^2/\G(2/3)\;,\\
K\big((3+\sqrt{-7})/8\big)&=(3^{1/2}/8)\G(1/7)\G(2/7)\G(4/7)/\pi\;.
\end{align*}
In addition:
\begin{align*}
  K\big(2\sqrt{\vphantom|\smash[t]{2-\sqrt{3}}}\big)&=2^{-4/3}3^{-1/4}(2+\sqrt{3})^{1/2}e^{-2i\pi/12}\G(1/3)^2/\G(2/3)\;,\\
  K\big(4\sqrt{\vphantom|\smash[t]{8-3\sqrt{7}}}\big)&=(1/16)(6/\sqrt{7}-2\sqrt{-1})^{1/2}(8+3\sqrt{7})^{1/2}\G(1/7)\G(2/7)\G(4/7)/\pi\;.
\end{align*}

At the same time, the evaluations of this corollary are certainly classical
(or easily obtainable by classical methods), so for instance the last one
combined with the continued fraction given by Theorem~\ref{thm:EK} gives an
irrationality measure for $\CS(-7)$ which could have been obtained long ago,
however superseded by one we give below.

Note that although there are infinitely many singular value evaluations of
$K(k)$ and $E(k)$ in terms of gamma quotients, only the first five above
give rise to rational and convergent continued fractions for $E(k)/K(k)$
(numbered (4.1) to (4.5) in Table~\ref{tablerat} below), while the last two give rise to rational but
divergent CFs (4.6) and~(4.7).

\section{Modular Hypergeometric Evaluations}

Recall that to obtain the rapidly convergent CF for $2^{1/3}$ from
Laguerre's, we chose $z=-3/128$. The crucial observation which will lead to
our results is that $-1/z=128/3$ is a \emph{CM value} of a Hauptmodul for the
$(3,3,\infty)$ triangle group corresponding to the hypergeometric function
${}_2F_1(1/2,5/6;1;z)$. By using this interpretation, we can compute
the quantities $T_i(1/2,5/6;128/3)$ for $i=0$ and $i=1$. This will first, prove
the validity of formula for the limit of the CF, and second, give ideas to find
further examples of the same kind by using other CM values and other
Hauptmoduln.

\smallskip

\subsection{Introduction}

\smallskip

We recalled above the famous modular hypergeometric evaluation giving
$E_6^{1/6}(\tau)$. In \cite{Beu-Coh}, for now unpublished, a large
number of similar evaluations are given corresponding to the nine
noncocompact arithmetic triangle groups. For each of these groups,
explicit Hauptmoduln (such as $1728/(1728-j(\tau))$) are given, as well as
all the rational CM evaluations of these Hauptmoduln. In particular,
one can define a Hauptmodul $R_1(\tau)$ for the triangle group $(3,3,\infty)$,
and note that $R_1((-1+3\sqrt{-3})/2)=128/3$, thus giving an explanation for
the occurrence of this number. All the formulas will be given explicitly
below, but for now note that to apply Corollary \ref{cor:1} or Theorem
\ref{thm:main2} we need hypergeometric functions ${}_2F_1(a,b;c;z)$ with
$c=2a$, and we also need $|z|<1$. The most important noncocompact triangle
groups for which this occurs are $(p,p,\infty)$ with $p=3$, $4$, $6$,
and $\infty$, with respective Hauptmoduln denoted $R_1$, $R_2$, $R_3$, and
$S_4$ in loc.~cit, and to uniformize we will set $R_4=1-S_4$, and
with this convention the argument $z$ of the hypergeometric function
will always be $1/(1-R_N(\tau))$, hence the argument $z$ of $T_n$
will be $1-R_N(\tau)$.

\smallskip

\begin{remarks}\begin{enumerate}
\item Note that the triangle groups $(2,\infty,\infty)$ and $(3,\infty,\infty)$
with respective Hauptmoduln denoted by $S_2$, $S_3$ also give evaluations of
${}_2F_1(a,b;c;z)$ with $c=2a$, but although they do produce continued
fractions, these are not in our family and do not give irrationality results,
so we will not consider them.
\item We have used the expression of the limit in terms of the functions
  $T_n$. If instead we used the expression given by Theorem \ref{thm:alt},
  we would use evaluations for the triangle groups $(2,3,\infty)$,
  $(2,4,\infty)$, $(2,6,\infty)$, and $(\infty,\infty,\infty)$ with
  respective Hauptmoduln denoted $J_N$ for $N=1$, $2$, $3$, and $4$,
  and we have $J_N=4R_N(1-R_N)=1-(2R_N-1)^2$.
  We would obtain exactly the same continued fractions and the same
  results since the two hypergeometric expressions are linked by a
  quadratic transformation formula, see below.
  \end{enumerate}
\end{remarks}

\smallskip

\subsection{List of Hauptmoduln and Modular Functions}

\smallskip

All of our examples will be in levels $1$, $2$, $3$, or $4$. To make this
paper self-contained, we give the definitions of all the functions that we
need. If a function has a single index (such as $R_1(\tau)$), this is the
level. If it has two indices (such as $E_{2,4}(\tau)$), the first index is
the level, and the second is the weight. We use standard notation for
modular forms, in particular $E_{2k}$ for Eisenstein series, $\eta$ for
the Dedekind eta function, and $\theta$ for the standard univariate
theta function of weight $1/2$ on $\G_0(4)$.

\smallskip
\noindent
{\bf Eisenstein Series:}
In levels $N=2$, $3$, and $4$ we define
$$E_{N,2}(\tau)=\dfrac{NE_2(N\tau)-E_2(\tau)}{N-1}\text{\quad and\quad}
E_{N,4}(\tau)=\dfrac{N^2E_4(N\tau)-E_4(\tau)}{N^2-1}\;,$$
and, in addition,
$$E_{3,3}(\tau)=E_{3,4}(\tau)/E_{3,2}(\tau)^{1/2} \quad\text{and}\quad
G_{4,2}(\tau)=4E_2(4\tau)-4E_2(2\tau)+E_2(\tau)\;.$$

\smallskip
\noindent
{\bf Auxiliary Functions and Hauptmoduln:}
We set
\begin{gather*}
  F_{1,\pm}(\tau)=E_6(\tau)\pm24\sqrt{-3}\eta^{12}(\tau) \quad\text{and}\quad
  R_1(\tau)=\dfrac{F_{1,+}(\tau)}{48\sqrt{-3}\eta^{12}(\tau)}\;; \\
  F_{2,\pm}(\tau)=E_{2,4}(\tau)\pm16\sqrt{-1}(\eta(\tau)\eta(2\tau))^4 \quad\text{and}\quad
  R_2(\tau)=\dfrac{F_{2,+}(\tau)}{32\sqrt{-1}(\eta(\tau)\eta(2\tau))^4}\;; \\
  F_{3,\pm}(\tau)=E_{3,3}(\tau)\pm6\sqrt{-3}(\eta(\tau)\eta(3\tau))^3 \quad\text{and}\quad
  R_3(\tau)=\dfrac{F_{3,+}(\tau)}{12\sqrt{-3}(\eta(\tau)\eta(3\tau))^3}\;; \\
  R_4(\tau)=\dfrac{G_{4,2}(\tau)+E_{4,2}(\tau)}{G_{4,2}(\tau)-E_{4,2}(\tau)}\;.
\end{gather*}

\smallskip

\subsection{Functional Modular Evaluations}

\smallskip

Although these formulas are certainly known, we note that all the
hypergeometric functional modular evaluations can be deduced from a general
theorem of F.~Beukers, see \cite{Beu-Coh}.

\begin{theorem}\label{thm:modeval}
  For each triangle group $(p,p,\infty)$ given below and for all
  $\tau$ in a suitable fundamental domain (given in \cite{Beu-Coh}) of that
  group, we have the following evaluations:
  \begin{align*}
    (3,3,\infty) :&\;\; {}_2F_1(1/2,5/6;1;1/(1-R_1(\tau)))=(F_{1,-}^{1/2}/F_{1,+}^{1/3})(\tau)\;,\\
    (4,4,\infty) :&\;\; {}_2F_1(1/2,3/4;1;1/(1-R_2(\tau)))=(F_{2,-}^{1/2}/F_{2,+}^{1/4})(\tau)\;,\\
    (6,6,\infty) :&\;\; {}_2F_1(1/2,2/3;1;1/(1-R_3(\tau)))=(F_{3,-}^{1/2}/F_{3,+}^{1/6})(\tau)\;,\\
    (\infty,\infty,\infty) :&\;\; {}_2F_1(1/2,1/2;1;1/(1-R_4(\tau)))=\theta^2(\tau)\;.
  \end{align*}
\end{theorem}

Beukers' proof is general and elegant, but once obtained, it is immediate
to \emph{prove} the above identities directly by showing that both sides
satisfy a linear differential equation of order two with the same initial
conditions.

\smallskip

Note that all the CM points $\tau$ that we will use are in the suitable
fundamental domains, but if they were not, by modularity we would simply
multiply by an automorphy factor.

\smallskip

We have mentioned above that we could also use the other hypergeometric
ratio formula for the limit $L_0$ and that it corresponds to the triangle
groups $(2,3,\infty)$, $(2,4,\infty)$, $(2,6,\infty)$, and
$(\infty,\infty,\infty)$. The relation between these expressions is a
consequence of one of the \emph{quadratic transformations} of hypergeometric
functions, here
$${}_2F_1(a,b;2a;z)=(1-z/2)^{b-2a}(1-z)^{a-b}{}_2F_1(a-b/2,a-b/2+1/2,a+1/2,(z/(2-z))^2)\;.$$
In particular:

\begin{align*}
  {}_2F_1(1/2,5/6;1;z)&=(1-z/2)^{-1/6}(1-z)^{-1/3}{}_2F_1(1/12,7/12;1;(z/(2-z))^2)\;,\\
  {}_2F_1(1/2,3/4;1;z)&=(1-z/2)^{-1/4}(1-z)^{-1/4}{}_2F_1(1/8,5/8;1;(z/(2-z))^2)\;,\\
  {}_2F_1(1/2,2/3;1;z)&=(1-z/2)^{-1/3}(1-z)^{-1/6}{}_2F_1(1/6,2/3;1;(z/(2-z))^2)\;,\\
  {}_2F_1(1/2,1/2;1;z)&=(1-z/2)^{-1/2}{}_2F_1(1/4,3/4;1;(z/(2-z))^2)\;.
\end{align*}

Note that if $z=1/(1-R_N(\tau))$ we have $(z/(2-z))^2=1/(2R_N(\tau)-1)^2$.

For instance, the analogue of the $(3,3,\infty)$ hypergeometric modular
evaluation given above is the $(2,3,\infty)$ Fricke evaluation
${}_2F_1(1/12,7/12;1;1/(1-j(\tau)/1728))=E_6^{1/6}(\tau)$ already mentioned
above.

\smallskip

\subsection{List of CM Points and Values}

\smallskip

In view of Theorems \ref{thm:main2} and \ref{thm:modeval}, to have rational continued fractions
we thus need to have $1-R_N(\tau)$ (or, equivalently, $R_N(\tau)$) to be of
the form $1/2+\sqrt{r}$ for some rational $r$. The list of such
$R_N(\tau)$ is finite, and corresponds to a generalization of the finiteness
of imaginary quadratic fields of class number $1$ (it is exactly this for
$N=1$), and is given in \cite{Beu-Coh}. We give the complete list in
Table~\ref{tablerat} together with the following additional information.

Thanks to Theorem \ref{thm:main2}, each rational value of $(2R_N(\tau)-1)^2$
gives rise to a continued fraction of our family, i.e., with
$a(n)=A(n-1)$ for $n\ge2$ and $b(n)=-K(Dn-1)(D(n-1)+1)$ for $n\ge1$, and
we give $A$ and $K$. Since changing $(A,K)$ into $(Am,Km^2)$ does not
change the CF, we choose $K$ to be squarefree and $A>0$, making the pair $(A,K)$
unique. These values in turn determine the speed of convergence $E$ of the
continued fraction so that $L-p(n)/q(n)\sim C_1/E^n$ and
$\log(|q'(n)L-p'(n)|)\sim-n\log(|E|)/2$, with the notation
of Theorem \ref{thm:main2}. Note that for five of our CM evaluations we have
$A^2-4KD^2<0$, so the theorem is not applicable, and indeed the corresponding
CFs do not converge.

\smallskip
We will see below in Theorem \ref{thm:denom} that the denominators of $p'(n)$
and $q'(n)$ divide a certain $d_D^*(n)$ (possibly multiplied by an unimportant
factor), and Proposition \ref{prop:dDstar} will tell us that
$\log(d_D^*(n))\sim m_D^*\cdot n$ as $n\to\infty$, where
$m_2^*=2$, $m_3^*=2.093\cdots$, $m_4^*=2.429\cdots$,
and $m_6^*=3.279\cdots$. It follows that whenever
the number in column $\log(|E|)/2$ is larger than $m_D^*$, the corresponding
continued fraction will converge to a limit for which it will trivially be
possible to give an irrationality measure, and in this case we put a
``Y'' in the last column.

\smallskip

Note that the 44 CM values in Table \ref{tablerat} are the same
as those of a similar table given in \cite{Coh-Gui}
used to obtain Ramanujan-type rational hypergeometric formulas for $1/\pi$.

\smallskip

Note also that making this table requires very little work, and immediately
tells us when we are going to obtain an irrationality measure. Of course
the main difficulty which remains is to know of what number we found an
irrationality measure of, in other words to compute the limit of the
continued fractions, and this will be done using Theorem \ref{thm:main2}.

\begin{table}[H]
  \caption{Rational Values of $(2R_N(\tau)-1)^2$}
  \label{tablerat}
\begin{center}
%\tiny
%\scriptsize
%\footnotesize
\small
  \begin{tabular}{|c|c|c|c|c|c|c|c|c|c|}
    \hline
    Tag & $N$ & $D$ & $\tau$ & $(2R_N(\tau)-1)^2$ & $A$ & $K$ & $\log(|E|)/2$ & $m^*_D$ & Irr? \\
    \hline\hline
    (1.1) & $1$ & $6$ & $2\sqrt{-1}$ & $-1323/8$ & $378$ & $-6$ & $3.248$ & $3.279$ & \\
    (1.2) & $1$ & $6$ & $\sqrt{-2}$ & $-98/27$ & $56$ & $-6$ & $1.400$ & $3.279$ & \\
    (1.3) & $1$ & $6$ & $\sqrt{-3}$ & $-121/4$ & $66$ & $-1$ & $2.406$ & $3.279$ & \\
    (1.4) & $1$ & $6$ & $(-1+3\sqrt{-3})/2$ & $64009/9$ & $1012$ & $1$ & $5.128$ & $3.279$ & Y \\
    (1.5) & $1$ & $6$ & $\sqrt{-7}$ & $-614061/64$ & $10773/2$ & $-21$ & $5.278$ & $3.279$ & Y \\
    (1.6) & $1$ & $6$ & $(-1+\sqrt{-7})/2$ & $189/64$ & $189/2$ & $21$ & $1.137$ & $3.279$ & \\
    (1.7) & $1$ & $6$ & $(-1+\sqrt{-11})/2$ & $539/27$ & $308$ & $33$ & $2.177$ & $3.279$ & \\
    (1.8) & $1$ & $6$ & $(-1+\sqrt{-19})/2$ & $513$ & $2052$ & $57$ & $3.813$ & $3.279$ & Y \\
    (1.9) & $1$ & $6$ & $(-1+\sqrt{-43})/2$ & $512001$ & $97524$ & $129$ & $7.266$ & $3.279$ & Y \\
    (1.10) & $1$ & $6$ & $(-1+\sqrt{-67})/2$ & $85184001$ & $1570212$ & $201$ & $9.823$ & $3.279$ & Y \\
    (1.11) & $1$ & $6$ & $(-1+\sqrt{-163})/2$ & $151931373056001$ & $3270840804$ & $489$ & $17.020$ & $3.279$ & Y \\
    \hline
    (2.1) & $2$ & $4$ & $\sqrt{-1}$ & $-49/32$ & $14$ & $-2$ & $1.040$ & $2.429$ & \\
    (2.2) & $2$ & $4$ & $(-1+3\sqrt{-1})/2$ & $49$ & $56$ & $1$ & $2.634$ & $2.429$ & Y \\
    (2.3) & $2$ & $4$ & $(-1+5\sqrt{-1})/2$ & $25921$ & $1288$ & $1$ & $5.775$ & $2.429$ & Y \\
    (2.4) & $2$ & $4$ & $3\sqrt{-2}/2$ & $-2400$ & $960$ & $-6$ & $4.585$ & $2.429$ & Y \\
    (2.5) & $2$ & $4$ & $(-1+\sqrt{-3})/2$ & $25/16$ & $10$ & $1$ & $0.693$ & $2.429$ & \\
    (2.6) & $2$ & $4$ & $(-1+\sqrt{-5})/2$ & $5$ & $40$ & $5$ & $1.444$ & $2.429$ & \\
    (2.7) & $2$ & $4$ & $\sqrt{-6}/2$ & $-8$ & $32$ & $-2$ & $1.763$ & $2.429$ & \\
    (2.8) & $2$ & $4$ & $(-1+\sqrt{-7})/2$ & $4225/256$ & $65/2$ & $1$ & $2.079$ & $2.429$ & \\
    (2.9) & $2$ & $4$ & $(-3+\sqrt{-7})/4$ & $175/256$ & $35/2$ & $7$ & $-$ & $-$ & \\
    (2.10) & $2$ & $4$ & $\sqrt{-10}/2$ & $-80$ & $160$ & $-5$ & $2.887$ & $2.429$ & Y \\
    (2.11) & $2$ & $4$ & $(-1+\sqrt{-13})/2$ & $325$ & $520$ & $13$ & $3.584$ & $2.429$ & Y \\
    (2.12) & $2$ & $4$ & $\sqrt{-22}/2$ & $-9800$ & $1120$ & $-2$ & $5.288$ & $2.429$ & Y \\
    (2.13) & $2$ & $4$ & $(-1+\sqrt{-37})/2$ & $777925$ & $42920$ & $37$ & $7.475$ & $2.429$ & Y \\
    (2.14) & $2$ & $4$ & $\sqrt{-58}/2$ & $-96059600$ & $422240$ & $-29$ & $9.883$ & $2.429$ & Y \\
    \hline
    (3.1) & $3$ & $3$ & $(-2+\sqrt{-2})/3$ & $25/27$ & $10$ & $3$ & $-$ & $-$ & \\
    (3.2) & $3$ & $3$ & $2\sqrt{-3}/3$ & $-25/2$ & $30$ & $-2$ & $1.975$ & $2.093$ & \\
    (3.3) & $3$ & $3$ & $(-1+\sqrt{-3})/2$ & $25/9$ & $10$ & $1$ & $1.099$ & $2.093$ & \\
    (3.4) & $3$ & $3$ & $(-3+5\sqrt{-3})/6$ & $81$ & $54$ & $1$ & $2.887$ & $2.093$ & Y \\
    (3.5) & $3$ & $3$ & $(-3+7\sqrt{-3})/6$ & $3025$ & $330$ & $1$ & $4.700$ & $2.093$ & Y \\
    (3.6) & $3$ & $3$ & $\sqrt{-6}/3$ & $-1$ & $6$ & $-1$ & $0.881$ & $2.093$ & \\
    (3.7) & $3$ & $3$ & $(-5+\sqrt{-11})/6$ & $11/27$ & $22$ & $33$ & $-$ & $-$ & \\
    (3.8) & $3$ & $3$ & $\sqrt{-15}/3$ & $-121/4$ & $33$ & $-1$ & $2.406$ & $2.093$ & Y \\
    (3.9) & $3$ & $3$ & $(-3+\sqrt{-15})/6$ & $5/4$ & $15$ & $5$ & $0.481$ & $2.093$ & \\
    (3.10) & $3$ & $3$ & $(-3+\sqrt{-51})/6$ & $17$ & $102$ & $17$ & $2.094$ & $2.093$ & Y \\
    (3.11) & $3$ & $3$ & $(-3+\sqrt{-123})/6$ & $1025$ & $1230$ & $41$ & $4.159$ & $2.093$ & Y \\
    (3.12) & $3$ & $3$ & $(-3+\sqrt{-267})/6$ & $250001$ & $28302$ & $89$ & $6.908$ & $2.093$ & Y \\
    \hline
    (4.1) & $4$ & $2$ & $\sqrt{-1}/2$ & $9$ & $12$ & $1$ & $1.76$ & $2$ & \\
    (4.2) & $4$ & $2$ & $\sqrt{-1}/4$ & $9/8$ & $6$ & $2$ & $0.347$ & $2$& \\
    (4.3) & $4$ & $2$ & $\sqrt{-2}/4$ & $2$ & $8$ & $2$ & $0.881$ & $2$ & \\
    (4.4) & $4$ & $2$ & $(-1+\sqrt{-3})/4$ & $-3$ & $12$ & $-3$ & $1.317$ & $2$ & \\
    (4.5) & $4$ & $2$ & $(-1+\sqrt{-7})/4$ & $-63$ & $84$ & $-7$ & $2.769$ & $2$ & Y \\
    (4.6) & $4$ & $2$ & $(-1+\sqrt{-3})/8$ & $3/4$ & $6$ & $3$ & $-$ & $-$ & \\
    (4.7) & $4$ & $2$ & $(-1+\sqrt{-7})/16$ & $63/64$ & $21/2$ & $7$ & $-$ & $-$ & \\
    \hline
  \end{tabular}
\end{center}
\end{table}

%\vfill\eject

\smallskip

\subsection{Computing the Examples}\label{sec:excomp}

\smallskip

We now explain how to compute our examples, in other words the
limits of the continued fractions. For each triangle group,
we have seen above modular hypergeometric evaluations of the form
${}_2F_1(a,b;2a;t(\tau))=f(\tau)$, where $t(\tau)$ is some Hauptmodul,
a modular function of weight $0$, and $f(\tau)$ is a modular function
(i.e., with possible poles) of weight $1$.

For a CM value of $\tau$ such that $t(\tau)\in\Q$ (or more generally
because of our special family, $(2t(\tau)-1)^2\in\Q$), we compute a
\emph{basic period} $\Om(\tau)$, which can be taken to be
$e^{i\pi/4}\eta(\tau)^2$ for instance (the $e^{i\pi/4}$ factor is irrelevant
but makes the value real in many cases),
expressible thanks to the theorem of Chowla--Selberg and
generalizations as the product of an algebraic number times a gamma quotient
to some fractional power.
By CM theory, we know that the value at $\tau$ of a modular function of weight
$k$ with algebraic Fourier coefficients will be equal to $\Om(\tau)^k$ times
an algebraic number, so for our above evaluation, $f(\tau)/\Om(\tau)$ will
be algebraic. This is also true for the non-holomorphic Eisenstein series
of weight $2$, $E_2^*(\tau)=E_2(\tau)-3/(\pi\Im(\tau))$, in other words
$E_2^*(\tau)/\Om(\tau)^2$ is algebraic.

Now that we have computed $f(\tau)$, to use Corollary \ref{cor:1} we also
need to compute ${}_2F_1(a+1,b+1;2a+2;t(\tau))$. Thanks to Lemma \ref{lem:1},
for this it suffices to compute
${}_2F_1'(a,b;2a;t(\tau))=D(f)(\tau)/D(t)(\tau)$ (where
$D=(2\pi i)^{-1}d/d\tau=q\,d/dq$). Now $D(t)(\tau)$ is a modular function of
weight $2$, so $D(t)(\tau)/\Om(\tau)^2$ is an algebraic number. The Serre
derivative
$D_{E_2}(f)(\tau)=D(f)(\tau)-(E_2(\tau)/12)f(\tau)$ is a modular
function of weight $3$, so $D_{E_2}(f)(\tau)/\Om(\tau)^3$ is an algebraic number.
Finally, as already mentioned $E_2(\tau)=E_2^*(\tau)+3/(\pi\Im(\tau))$
and $E_2^*(\tau)/\Om(\tau)^2$ is an algebraic number. Using all of this
allows us to compute ${}_2F_1(a+1,b+1;2a+2;t(\tau))$, and thus,
thanks to Theorem \ref{thm:main2}, the limit of our continued fractions.

\section{List of CM Examples}

Since we have a large number of CM examples, it would be extremely tedious
for the reader to go through all of them one after the other. We will thus
explain in detail the computation of (1.4), the first ``Y'' in our table,
which will lead to the first known irrationality measure for $\CS(-3)$, and
only give the results for the others in the form of tables.

\smallskip

\subsection{Example: The CM Value $z=128/3$ for $(3,3,\infty)$}

\smallskip

We specialize the $(3,3,\infty)$ evaluation given above to
$\tau=(-3+3\sqrt{-3})/2$, for which $1-R_1(\tau)=128/3$.
We choose $\Om(\tau)=e^{i\pi/4}\eta(\tau)^2$, and
for notational simplicity, we omit the argument $\tau$. We find that:
\begin{align*}&\Om=3^{-19/12}\G(1/3)/\G(2/3)^2\;,\\
&R_1=-125/3\;,\quad D(R_1)=800\cdot3^{-5/6}\Om^2\;,\quad f=2^{25/6}3^{1/12}5^{-1}\Om\;,\\
&E_2^*=8\cdot3^{1/6}\Om^2\;,\quad D_{E_2}(f)=-119\cdot2^{7/6}3^{-3/4}\cdot5^{-2}\Om^3\;,\\
&{}_2F_1(1/2,5/6;1;3/128)=2^{25/6}3^{-3/2}5^{-1}\dfrac{\G(1/3)}{\G(2/3)^2}\;,\\
&{}_2F_1(3/2,11/6;3;3/128)=2^{85/6}3^{-3/2}5^{-2}\left(31\dfrac{\G(1/3)}{\G(2/3)^2}-240\dfrac{\G(2/3)}{\G(1/3)^2}\right)\;.
\end{align*}

%\vfill\eject

\begin{table}[H]
  \caption{Table of Continued Fractions}
  \label{tableall}
\begin{center}
  %  \scriptsize
  \small
  \begin{tabular}{|c|c||c|c|c|c|c|c|}
    \hline
    Tag & $L$ & $a_1$ & $A$ & $b_0$ & $K$ & $D$ & $\mu(L)$ \\
    \hline\hline
    (1.1) & $\CS(-4)$ & $15$ & $378$ & $132$ & $-6$ & $6$ & $-$ \\
    (1.2) & $\CS(-8)$ & $3$ & $56$ & $40$ & $-6$ & $6$ & $-$ \\
    (1.3) & $2^{1/3}\CS(-3)$ & $3$ & $66$ & $30$ & $-1$ & $6$ & $-$ \\
    (1.4) & $\CS(-3)$ & $31$ & $1012$ & $240$ & $1$ & $6$ & $5.548$ \\
    (1.5) & $\CS(-7)$ & $324$ & $10773$ & $3570$ & $-84$ & $6$ & $5.282$ \\
    (1.6) & $\CS(-7)$ & $12$ & $189$ & $105$ & $84$ & $6$ & $-$ \\
    (1.7) & $\CS(-11)$ & $15$ & $308$ & $176$ & $33$ & $6$ & $-$ \\
    (1.8) & $\CS(-19)$ & $75$ & $2052$ & $912$ & $57$ & $6$ & $14.294$ \\
    (1.9) & $\CS(-43)$ & $2367$ & $97524$ & $20640$ & $129$ & $6$ & $3.646$ \\
    (1.10) & $\CS(-67)$ & $30531$ & $1570212$ & $176880$ & $201$ & $6$ & $3.003$ \\
    (1.11) & $\CS(-163)$ & $40774227$ & $3270840804$ & $52186080$ & $489$ & $6$ & $2.478$ \\
    \hline
    (2.1) & $\CS(-4)$ & $1$ & $14$ & $12$ & $-2$ & $4$ & $-$ \\
    (2.2) & $12^{1/4}\CS(-4)$ & $3$ & $56$ & $48$ & $1$ & $4$ & $25.733$ \\
    (2.3) & $5^{-1/4}\CS(-4)$ & $41$ & $1288$ & $240$ & $1$ & $4$ & $3.453$ \\
    (2.4) & $6^{1/2}\CS(-8)$ & $36$ & $960$ & $1008$ & $-6$ & $4$ & $4.254$ \\
    (2.5) & $2^{1/3}\CS(-3)$ & $1$ & $10$ & $6$ & $1$ & $4$ & $-$ \\
    (2.6) & $\CS(-20)$ & $3$ & $40$ & $40$ & $5$ & $4$ & $-$ \\
    (2.7) & $2^{1/2}\CS(-24)$ & $2$ & $32$ & $48$ & $-2$ & $4$ & $-$ \\
    (2.8) & $\CS(-7)$ & $4$ & $65$ & $42$ & $4$ & $4$ & $-$ \\
    (2.10) & $\CS(-40)$ & $8$ & $160$ & $120$ & $-5$ & $4$ & $12.607$ \\
    (2.11) & $\CS(-52)$ & $23$ & $520$ & $312$ & $13$ & $4$ & $6.207$ \\
    (2.12) & $2^{1/2}\CS(-88)$ & $38$ & $1120$ & $528$ & $-2$ & $4$ & $3.700$ \\
    (2.13) & $\CS(-148)$ & $1123$ & $42920$ & $6216$ & $37$ & $4$ & $2.963$ \\
    (2.14) & $\CS(-232)$ & $8824$ & $422240$ & $22968$ & $-29$ & $4$ & $2.652$ \\
    \hline
    (3.2) & $2^{1/3}3^{1/2}\CS(-3)$ & $2$ & $30$ & $36$ & $-2$ & $3$ & $-$ \\
    (3.3) & $\CS(-3)$ & $1$ & $10$ & $6$ & $1$ & $3$ & $-$ \\
    (3.4) & $5^{1/6}\CS(-3)$ & $3$ & $54$ & $30$ & $1$ & $3$ & $7.272$ \\
    (3.5) & $3^{1/2}7^{-1/6}\CS(-3)$ & $13$ & $330$ & $126$ & $1$ & $3$ & $3.606$ \\
    (3.6) & $\CS(-24)$ & $1/2$ & $6$ & $12$ & $-1$ & $3$ & $-$ \\
    (3.8) & $\CS(-15)$ & $2$ & $33$ & $30$ & $-1$ & $3$ & $15.377$ \\
    (3.9) & $\CS(-15)$ & $2$ & $15$ & $15$ & $5$ & $3$ & $-$ \\
    (3.10) & $\CS(-51)$ & $7$ & $102$ & $102$ & $17$ & $3$ & $2598.6$ \\
    (3.11) & $\CS(-123)$ & $53$ & $1230$ & $492$ & $41$ & $3$ & $4.027$ \\
    (3.12) & $\CS(-267)$ & $827$ & $28302$ & $2670$ & $89$ & $3$ & $2.870$ \\
    \hline
    (4.1) & $\CS(-4)$ & $1$ & $12$ & $8$ & $1$ & $2$ & $-$ \\
    (4.2) & $\CS(-4)$ & $1$ & $6$ & $4$ & $2$ & $2$ & $-$ \\
    (4.3) & $\CS(-8)$ & $1$ & $8$ & $8$ & $2$ & $2$ & $-$ \\
    (4.4) & $2^{1/3}\CS(-3)$ & $1$ & $12$ & $12$ & $-3$ & $2$ & $-$ \\
    (4.5) & $\CS(-7)$ & $5$ & $84$ & $56$ & $-7$ & $2$ & $7.204$ \\
    \hline
  \end{tabular}
\end{center}
\end{table}

\vfill\eject

Using the theory explained in the previous sections, especially Corollary
\ref{cor:1}, we deduce our first theorem, which proves the validity of
our conjectural CF for $\CS(-3)$:

\begin{theorem}\label{thm:1}
  We have
  $$\CS(-3)=\left(\dfrac{\G(1/3)}{\G(2/3)}\right)^3=[[0,31,1012(n-1)],[240,-(6n-1)(6n-5)]]$$
with speed of convergence
$$\CS(-3)-\dfrac{p(n)}{q(n)}\sim\dfrac{3^{3/2}\CS(-3)}{(16+5\sqrt{10})^{4n}6^{-2n}}\;.$$
In addition, if we set $(p'(n),q'(n))=(p(n),q(n))/\prod_{1\le j\le n}(6j-5)$
we have
$$\log(|q'(n)\CS(-3)-p'(n)|)\sim -n\log((253+80\sqrt{10})/3)\;.$$
\end{theorem}

As mentioned, we will see below that this leads to the first known
irrationality measure for $\CS(-3)$.

\smallskip

\subsection{The Continued Fractions}

\smallskip

In Table \ref{tableall} we give a table of the CFs which are
obtained from the above hypergeometric evaluations using Corollary \ref{cor:1}
and Theorem \ref{thm:main2}.

Each CF is of the form explained above
$$L=[[0,a_1,A(n-1)],[b_0,-K(Dn-1)(D(n-1)+1)]]\;,$$
and thanks to Table \ref{tablerat}, we know the speed of convergence $E$
hence $\log(|E|)/2$, and when this is large enough, we can thus obtain
an irrationality measure, given in the last column. We do not include (2.9),
(3.1), (3.7), (4.6), and (4.7) since the corresponding CFs do not converge. We also
recall that the definition of $\CS(D)$ involves an exponent $w(D)/(2h(D))$
which is equal to $1/2$ when $h(D)=2$, which occurs for all CFs in levels
2 and 3, except those for $D=-3$, $-4$, $-7$, and $-8$.

Although we have only given $\log(|E|)/2$ and not $E$ itself, five of the above continued
fractions have $E$ rational, so we can use Ap\'ery-type techniques as explained in
\cite{Coh1} to obtain new CFs. It was in fact in this manner that
the rapidly convergent CF (1.5) for $\CS(-7)$ was first obtained.
The results are rather disappointing. First, all five are self-dual
if we use the fastest possible Ap\'ery acceleration. The Ap\'ery
accelerates of (2.1), (2.5), (2.8), and (4.2) give respectively
(1.1), (1.3), (1.5), and (2.1), while (3.3) does not simplify.
Using slower Ap\'ery techniques does give new CFs, but which are
not in our family and do not seem interesting, see Appendix I for a list.

\smallskip

\subsection{Unshifted Continued Fractions}\label{sec:unshift}

\smallskip

Recall that our initial CF for $\CS(-3)$ was initially conjectured by
\emph{shifting} (i.e., by changing $n$ to $n-1/2$) a rapidly convergent one
for $2^{1/3}$. The reader can play with all the CFs that we have found by
unshifting them (changing $n$ into $n+1/2$) and computing the corresponding
limits. These will be algebraic numbers of degree less than or equal to $6$
for $N=1$, $2$, and $3$, and M\"obius transforms of logarithms of (possibly
complex) algebraic numbers for $N=4$.

\smallskip

We give two examples. First, after M\"obius transformations the unshift
of (2.3) gives the CF (which can be obtained directly from Laguerre's by
choosing $a=1/4$ and $z=1/80$):
$$5^{1/4}=[[3/2,645,644(2n-1)],[-3,-(4n-1)(4n+1)]]$$
with speed of convergence in $E^{-n}$ with
$E=((1+\sqrt{5})/2)^{24}\approx103682$,
which, after analysis of the denominators as we will do below,
can be shown to lead to the effective irrationality measure
$\mu(5^{1/4})<2.7474$, which is essentially the same as the best known one.

\smallskip

Second, similarly the unshift of (4.4) gives the \emph{classical} CF
$$\pi/\sqrt{3}=[[0,3(2n-1)],[6,3n^2]]$$
which leads to the irrationality measure $\mu(\pi/\sqrt{3})<8.310$, which
is not as good as the best known one for this number.

\smallskip

This finishes the analytic part of the paper. In order to prove irrationality
and obtain the irrationality measures given above, we must now bound the
denominators of the partial quotients of the continued fractions, which
we will do in the next arithmetic part of the paper. Once this is done,
we will have proved the following theorem:
\begin{theorem}\label{thm:irrall}
  We have the following bounds on irrationality measures:
  \begin{alignat*}{3}
    \mu(\CS(-3))&<5.548\;,\quad&\mu(5^{1/6}\CS(-3))&<7.272\;,\quad&\mu(3^{1/2}7^{-1/6}\CS(-3))&<3.606\;,\\
    \mu(12^{1/4}\CS(-4))&<25.733\;,\quad&\mu(5^{-1/4}\CS(-4))&<3.453\;,\quad&\mu(\CS(-7))&<5.282\;,\\
    \mu(6^{1/2}\CS(-8))&<4.254\;,\quad&\mu(\CS(-15))&<15.377\;,\quad&\mu(\CS(-19))&<14.294\;,\\
    \mu(\CS(-40))&<12.607\;,\quad&\mu(\CS(-43))&<3.646\;,\quad&\mu(\CS(-51))&<2598.6\;,\\
    \mu(\CS(-52))&<6.207\;,\quad&\mu(\CS(-67))&<3.003\;,\quad&\mu(2^{1/2}\CS(-88))&<3.700\;,\\
    \mu(\CS(-123))&<4.027\;,\quad&\mu(\CS(-148))&<2.963\;,\quad&\mu(\CS(-163))&<2.478\;,\\
    \mu(\CS(-232))&<2.652\;,\quad&\mu(\CS(-267))&<2.870\;.&&
  \end{alignat*}
\end{theorem}

The first bound in this theorem can be compared with the non-rigorously established
value $\mu(\CS(-3))<13.418$ which follows from a numerical calculation in
\cite{D-BKZ22,Zud21} alluded to in Section~\ref{sec:intro}. As already
mentioned, these seem to be the first proved (and reasonable)
irrationality measures for quantities linked to gamma quotients.

\smallskip

\begin{remarks}
  \begin{enumerate}
  \item Note that $\CS(-4)$, $\CS(-8)$, and $\CS(-88)$ only appear
    multiplied by an irrational algebraic number, and that $\CS(-11)$ does not
    appear since the CF (1.7), being the only convergent CF involving
    $\CS(-11)$, does not converge sufficiently fast.
  \item From the irrationality measure for $\CS(-3)$ we deduce
    trivially that
    $$\mu(\G(1/3)/\G(2/3))<16.644\;.$$
  \item This list \emph{cannot} be improved without introducing
    new methods: indeed, first the list of CM values is \emph{complete},
    in other words there are exactly 51 rational CM evaluations, of which
    only 44 are useful and give continued fractions, only 39 of them are
    convergent, and among those only 20 give irrationality measures, and almost
    certainly no more since numerically the denominator bounds that we
    will prove below are asymptotically optimal.
  \end{enumerate}
\end{remarks}

\section{Proofs of Irrationality}
\label{sec:irr}

\smallskip

\subsection{An Explicit Formula for the Convergents}

\smallskip

Recall that the general continued fraction of our family has the shape
$${\cal C}=[[0,a_1,A(n-1)],[b_0,-K(Dn-1)(D(n-1)+1)]]\;.$$
We denote by $p(n)/q(n)$ its $n$th partial quotient, and define
$$(p_1(n),q_1(n))=(p(n),q(n))/\prod_{1\le j\le n}(D(j-1)+1)\;.$$
Both $v_n=p_1(n)$ and $v_n=q_1(n)$ satisfy the recursion
$(Dn+1)v_{n+1}=Anv_n-K(Dn-1)v_{n-1}$ with $p_1(0)=0$, $q_1(0)=1$,
$p_1(1)=b_0$, and $q_1(1)=a_1$, or equivalently, after dividing by $D$
and setting $B=1/D$ and $Z=A/D$:
$$(n+B)v_{n+1}=Znv_n+K(B-n)v_{n-1}\;.$$

Our first theorem is an explicit formula for $v_n$:

\begin{theorem}\label{thm:mainrec}
  We have $$v_{n+1}=\dfrac{P_n(B,Z,K)}{(B+1)_n}v_1+\dfrac{Q_n(B,Z,K)}{(B+1)_n}v_0\;,$$
  where $(a)_n$ denotes the rising Pochhammer symbol, and where if we set
  $$\Lambda_i(B)=(B-i)_i(B)_i=(B-i)_{2i}=\prod_{m=-i}^{i-1}(B+m)\;,$$
  we have
  \begin{align*}
    P_n=P_n(B,Z,K)&=\sum_{j=0}^{\lfloor n/2\rfloor}(-1)^jK^{n-j}(Z/K)^{n-2j}\frac{(n-j)!}{(n-2j)!}
    \sum_{i=0}^j\frac{(-1)^i(n-i)!}{i!^2(j-i)!}\Lambda_i(B)\;,\\
    Q_n=Q_n(B,Z,K)&=(B-1)\sum_{j=0}^{\lfloor(n-1)/2\rfloor}(-1)^jK^{n-j}(Z/K)^{n-2j-1}
\\ &\qquad\times \frac{(n-j-1)!}{(n-2j-1)!}
    \sum_{i=0}^j\frac{(-1)^i(n-i)!}{i!(i+1)!(j-i)!}\Lambda_i(B)\;.\\
  \end{align*}
\end{theorem}

\begin{proof}
  The essential difficulty is of course to find these formulas. Once written
  down explicitly as above, they can easily be checked by induction on $n$.
  However, we owe to the reader a short explanation of how these formulas
  were obtained. By homogeneity, we may assume that $K=1$. We observed that
  each coefficient in the expansion of $P_n$ in powers of $Z$ was a numerical
  factor of a polynomial in $B$; the sequence of the numerical factors was
  identified using the Online Encyclopedia of Integer Sequences \cite{OEIS},
  while the symmetry of the polynomials with respect to the involution
  $B\mapsto1-B$ helped to identify them as numerical multiples of the
  truncated hypergeometric sums
  \[
    {}_3F_2(B, \, 1-B, \, -j; \, 1, \, -n; 1).
    \]
    The same procedure was applied to $Q_n/(B-1)$, after noticing that $Q_n$
    is always divisible by $B-1$.
\end{proof}

\begin{remark}
The polynomials $P_n(B,Z,1)/(B+1)_n$ and $Q_n(B,Z,1)/(B+1)_n$ in variable $x=Z/2$ are particular instances of associated ultraspherical polynomials \cite[Section~3]{BI82}.
This circumstance however is of no help in our arithmetic analysis below.
\end{remark}

\smallskip

\subsection{Bounding the Denominators}

\smallskip

We must now analyze the arithmetic of $P_n/(B+1)_n$ and $Q_n/(B+1)_n$.
Although we could do the analysis in general, we will restrict to our
situation where $B=1/D$ and $D\in\{2,3,4,6\}$. We always assume implicitly
that $v_0$ and $v_1$ are integral. As usual, we denote by $\{x\}=x-\lfloor x\rfloor$
the fractional part of a real number $x$.

\begin{theorem}\label{thm:denom}
  Assume that $B=1/D$ and $D\in\{2,3,4,6\}$. Define
  $$d_D(n)=\lcm(Dj+1)_{1\le j\le n}\text{\quad and\quad}d_D^*(n)=d_D(n)\Big/\prod_{p\in\cP_n}p\;,$$
  where for $D=2$ we set $\cP_n=\emptyset$, and otherwise
\begin{multline}
  \cP_n=\bigg\{p \;\text{prime}:\max(\sqrt{2Dn},D)<p\le n, \; p\equiv-1\pmod D,
\\
\Big\{\frac{n+1-1/D}p\Big\}\ge\frac1D \;\text{and}\; \Big\{\frac{n+1/D}p\Big\}<1-\frac1D\bigg\}.
\label{prime-set}
\end{multline}
\begin{enumerate}
\item[\textup{(i)}] If $D(Z/K)=A/K\in\Z$ then $d^*_D(n)K^{-\lfloor(n+1)/2\rfloor}v_{n+1}\in\Z$.
\item[\textup{(ii)}] More generally, denote by $g$ the denominator of $D(Z/K)=A/K$.
  Assume that all the prime divisors of $g$ divide $D$, that $v_2(K)\ge2v_2(g)-2$, and
  that $v_p(K)\ge2v_p(g)-1$ for $p\ge3$.

  Then there exist arithmetic functions $e_p(n)$ such that $e_p(n)=O(\log(n))$ and
  $$\prod_{p\mid g}p^{e_p(n)}d^*_D(n)K^{-\lfloor(n+1)/2\rfloor}v_{n+1}\in\Z$$
  (since $D\in\{2,3,4,6\}$, we can have only $p=2$ and $p=3$).
\end{enumerate}
\end{theorem}

\begin{proof}
The individual terms in the expressions for $P_n/(B+1)_n$ and $Q_n/(B+1)_n$ can be written as
\begin{equation}
(Z/K)^{n-2j}\binom{n-j}{j}\binom{j}{i}\cdot B\cdot\frac{(B-i)_i}{i!}\cdot\frac{(n-i)!}{(B+i)_{n-i+1}}
\quad\text{for}\; 0\le i\le j\le \frac n2
\label{monom}
\end{equation}
and
\begin{equation}
(Z/K)^{n-2j-1}\binom{n-j-1}{j}\binom{j}{i}\cdot(B-1)\cdot\frac{(B-i)_{i+1}}{(i+1)!}
\cdot\frac{(n-i)!}{(B+i)_{n-i+1}}
\quad\text{for}\; 0\le i\le j\le \frac{n-1}2\;,
\label{monom2}
\end{equation}
multiplied by $(-1)^{i+j}K^{n-j}$. Since $j\le n/2$ we have $n-j\ge\lfloor (n+1)/2\rfloor$, so
$K^{-\lfloor (n+1)/2\rfloor}v_{n+1}$ is a $\Z$-linear combination of the above quantities.

For each prime $p$, we must find an upper bound on the $p$-adic valuation
of their denominators. Assume first that $p\nmid g$, the denominator of $DZ/K$.

Consider first the primes $p$ that divide $D$
(for us this is only for $p=2$ and/or $p=3$). The factors
\[
\frac{B\cdot(B-i)_i}{(B+i)_{n-i+1}}
\quad\text{and}\quad
\frac{(B-1)\cdot(B-i)_{i+1}}{(B+i)_{n-i+1}}
\]
are expressible in the form $D^{n-2i}C$ and $D^{n-2i-1}C$ with a rational $C$ involving no
prime $p\mid D$.

In particular, since $p\nmid g$ it follows that
$D(Z/K)$ is $p$-integral, so the expression
\[
(Z/K)^{n-2j}\frac{B\cdot(B-i)_i}{(B+i)_{n-i+1}}
=D^{2(j-i)}\cdot(DZ/K)^{n-2j}\cdot\frac{D^{2i-n}\cdot B\cdot(B-i)_i}{(B+i)_{n-i+1}}
\]
is $p$-integral for any $p\mid D$, hence so is the entire expression in \eqref{monom};
similarly, the $p$-integrality holds for the expression in \eqref{monom2}.

For primes $p\nmid D$, we first decompose the $B$-part of the terms in \eqref{monom} and \eqref{monom2} into the sum of partial fractions in~$B$ viewed as a variable:
\[
B\cdot\frac{(B-i)_i}{i!}\cdot\frac{(n-i)!}{(B+i)_{n-i+1}}
=\sum_{k=i}^n\frac{\rho_k}{B+k}\;,
\quad\text{where}\;
\rho_k=(-1)^{k+1}k\binom{k+i}{i}\binom{n-i}{k-i}\in\mathbb Z\;,
\]
and similarly
\[
(B-1)\cdot\frac{(B-i)_{i+1}}{(i+1)!}
\cdot\frac{(n-i)!}{(B+i)_{n-i+1}}
=\sum_{k=i}^n\frac{\tilde\rho_k}{B+k}\;,
\quad\text{where}\;
\tilde\rho_k=(-1)^k(k+1)\binom{k+i}{i+1}\binom{n-i}{k-i}\in\mathbb Z\;.
\]
This means that the $B$-expressions are $\mathbb Z$-linear combinations of $1/(B+k)$ with $k=1,2,\dots,n$; in particular, multiplication of those with
$d_D(n)$ makes them $p$-integral for $p\nmid D$. Since $p\nmid D$ and $p\nmid g$,
we also have that $Z/K=(DZ/K)/D$ is $p$-integral, so is the full expression.

For part (i) of the theorem, it remains to discuss the economical choice of $d_D^*(n)$ in
place of $d_D(n)$. First note that if $p\in\cP_n$ we have $(D-1)p\equiv1\pmod{D}$ and
$(D-1)p<Dn+1$, so $(D-1)p$ divides $d_D(n)=\lcm(Dk+1)_{1\le k\le n}$. Thus, it
follows from the partial-fraction expansions that it is sufficient to check that,
for each $p\in\cP_n$, the $p$-adic orders of the rational numbers
\[
\frac{1}{Dk+1}\binom{k+i}{i}\binom{n-i}{k-i}
\quad\text{and}\quad
\frac{1}{Dk+1}\binom{k+i}{i+1}\binom{n-i}{k-i}\;,
\quad\text{where}\; i\le k\le n\;,
\]
are non-negative. Since $p\in\cP_n$ implies $p^2>2Dn>Dk+1$, we have
$v_p(Dk+1)\le1$, so this will be a consequence of the following technical
lemma:

\begin{lemma}
\label{prime-saving}
Fix non-negative integers $i\le n$ and a prime $p\equiv-1\pmod D$ satisfying
$\sqrt{2Dn}<p\le n$. Let $k$ be an integer with
$i\le k\le n$ such that $p\mid Dk+1$.
\begin{enumerate}
\item[(1)] If both $\binom{k+i}{i}$ and $\binom{n-i}{k-i}$ are not divisible
  by $p$ then either $\{(n+1-1/D)/p\}<1/D$ or $\{(n+1/D)/p\}\ge1-1/D$.
\item[(2)] If $p\ne D-1$ then if both $\binom{k+i}{i+1}$ and
  $\binom{n-i}{k-i}$ are not divisible by $p$ the same conclusion holds.
\end{enumerate}
\end{lemma}

\begin{proof} If $p>\sqrt m$ we evidently have $v_p(m!)=\lfloor m/p\rfloor$,
  hence if $p>\sqrt{a+b}$ we have
  $$v_p\binom{a+b}{b}=\lfloor (a+b)/p\rfloor-\lfloor a/p\rfloor-\lfloor b/p\rfloor=\lfloor \{a/p\}+\{b/p\}\rfloor\;.$$
  It follows that the binomial coefficient $\binom{a+b}{b}$ is not divisible
  by $p$ if and only if $\{a/p\}+\{b/p\}<\nobreak1$.

  Since $p\equiv-1\pmod{D}$ and $Dk+1\equiv0\pmod p$, it follows that
  $k\equiv-1/D\equiv((D-1)p-1)/D\pmod{p}$, so $\{k/p\}=1-1/D-1/(Dp)$.

  \smallskip
  
  (1) It follows that $\binom{k+i}{i}$ is not divisible by $p$ if and only if
  $\{i/p\}<1/D+1/(Dp)=((p+1)/D)/p$, hence
  $\{i/p\}\le ((p+1)/D-1)/p=1/D-(1-1/D)/p$.
  On the other hand,
\begin{align*}
v_p\binom{n-i}{k-i}
&=\Big\lfloor\frac np-\frac ip\Big\rfloor
-\Big\lfloor\frac np-\frac kp\Big\rfloor
-\Big\lfloor\frac kp-\frac ip\Big\rfloor
\\
&=\Big\lfloor\Big\{\frac np\Big\}+1-\Big\{\frac ip\Big\}\Big\rfloor
-\Big\lfloor\Big\{\frac np\Big\}+1-\Big\{\frac kp\Big\}\Big\rfloor
-\Big\lfloor\Big\{\frac kp\Big\}-\Big\{\frac ip\Big\}\Big\rfloor\;.
\end{align*}
We have $\{k/p\}=1-1/D-1/(Dp)$, and since $\{i/p\}<1/D+1/(Dp)$
and $D\ge3$, it follows that $\{i/p\}\le\{k/p\}$, so
$\lfloor \{k/p\}-\{i/p\}\rfloor=0$. Thus,
$$v_p\binom{n-i}{k-i}
=\Big\lfloor\Big\{\frac np\Big\}+1-\Big\{\frac ip\Big\}\Big\rfloor
-\bigg\lfloor\Big\{\frac np\Big\}+\frac{1}D+\frac1{Dp}\bigg\rfloor\;.
$$
This expression is equal to $0$ if and only if both integer parts are
equal to $1$, or both are equal to $0$. Recall the trivial fact that
if $0<\alpha<1$ then $\{(m+\alpha)/p\}=\{m/p\}+\alpha/p$. Thus, if both are
equal to $1$ we have $\{(n+1/D)/p\}=\{n/p\}+1/(Dp)\ge1-1/D$, while if both
are equal to $0$, we have $\{n/p\}<\{i/p\}\le 1/D-(1-1/D)/p$, hence
$\{(n+(1-1/D))/p\}<1/D$, proving (1).

\smallskip

(2) First note that since $Dk\equiv-1\pmod{p}$ we have $p\nmid k$, so
$\{(k-1)/p\}=1-1/D-1/(Dp)-1/p$. Thus as above, $\binom{k+i}{i+1}$ is not
divisible by $p$ if and only if
$\{(i+1)/p\}<((p+1)/D)/p+1/p$, hence $\{(i+1)/p\}\le 1/D+1/(Dp)$.
On the other hand, similarly to (1) we can write
$$v_p\binom{n-i}{k-i}=\Big\lfloor\Big\{\frac{n+1}p\Big\}+1-\Big\{\frac {i+1}p\Big\}\Big\rfloor
-\Big\lfloor\Big\{\frac {n+1}p\Big\}+1-\Big\{\frac {k+1}p\Big\}\Big\rfloor
-\Big\lfloor\Big\{\frac {k+1}p\Big\}-\Big\{\frac {i+1}p\Big\}\Big\rfloor\;.$$
Note that we cannot have $k\equiv-1\pmod{p}$, otherwise since
$Dk\equiv-1\pmod{p}$ we have $D\equiv1\pmod{p}$ so $p=D-1$ since
$p\equiv-1\pmod{D}$, which is excluded. Thus
$\{(k+1)/p\}=\{k/p\}+1/p=1-1/D-1/(Dp)+1/p$. As above, we have
$\{(i+1)/p\}\le 1/D+1/(Dp)<\{(k+1)/p\}$, so
$\lfloor\{(k+1)/p\}-\{(i+1)/p\}\rfloor=0$.

We check that if $n\equiv-1\pmod{p}$ we have $\{(n+1/D)/p\}\ge1-1/D$, so
we may assume that $n\not\equiv-1\pmod{p}$, so $\{(n+1)/p\}=\{n/p\}+1/p$.
Thus,
$$v_p\binom{n-i}{k-i}
=\Big\lfloor\Big\{\frac np\Big\}+1+\frac1p-\Big\{\frac{i+1}p\Big\}\Big\rfloor
-\bigg\lfloor\Big\{\frac np\Big\}+\frac{1}D+\frac1{Dp}\bigg\rfloor\;.
$$
If both integer parts are equal to $1$ we have, as in (1), \
$\{(n+1/D)/p\}\ge1-1/D$. If both are equal to $0$, we have
$\{n/p\}<\{(i+1)/p\}-1/p\le 1/D+1/(Dp)-1/p$. As in (1), it follows that
$\{(n+1-1/D)/p\}<1/D$, proving (2).
\end{proof}

We have thus proved that when $p\nmid g$, the expression
$d^*(n)K^{-\lfloor (n+1)/2\rfloor}v_{n+1}$ is $p$-integral.

For part (ii), we now assume that $p\mid g$, so that by assumption $p\mid D$, and
consider again the above expression (after dividing by $K^{\lfloor (n+1)/2\rfloor}$):
\[
K^{\lfloor n/2\rfloor-j}(Z/K)^{n-2j}\cdot\frac{B\,(B-i)_i}{i!}\cdot\frac{(n-i)!}{(B+i)_{n-i+1}}\;,
\quad\text{where}\; 0\le i\le j\le \frac n2\;,
\]
and the similar one for $Q_n$. Since $B=1/D$ and $p\mid D$, we have
$$v_p(B(B-i)_i/(B+i)_{n-i+1})=(-i-1+n-i+1)v_p(D)=(n-2i)v_p(D)\;.$$
On the other hand, $v_p(m!)=(m-s_p(m))/(p-1)$, where $s_p(m)$ is the sum of digits of~$m$ in base~$p$, so $v_p((n-i)!/i!)=(n-2i-s_p(n-i)+s_p(i))/(p-1)=(n-2i)/(p-1)+O(\log(n))$.
Writing $(Z/K)^{n-2j}=(DZ/K)^{n-2j}D^{2j-n}$, it follows that the $p$-adic valuation
of the above expression is equal to
$$2(j-i)v_p(D)+(n-2i)/(p-1)+(n/2-j)(v_p(K)-2v_p(g))+O(\log(n))\;.$$
if $p=2$, we have $v_p(K)-2v_p(g)\ge-2$, so this is greater than or equal to
$2(j-i)v_p(D)+2(j-i)+O(\log(n))\ge O(\log(n))$ since $i\le j$.
If $p\ge3$, we have $v_p(K)-2v_p(g)\ge-1\ge-2/(p-1)$, so this is greater than or
equal to $2(j-i)v_p(D)+2(j-i)/(p-1)+O(\log(n))\ge O(\log(n))$, finishing the
proof of the theorem.
\end{proof}

\smallskip

\begin{remarks}\label{rem:aps}
  \begin{enumerate}\item
    We introduced the condition $p>\sqrt{2Dn}$ in the definition of
  $\cP_n$ to ensure that $v_p(Dk+1)\le1$ and so as to give a simple expression
  for the valuation of the binomial coefficients, but numerics suggest that this
  condition is unnecessary, as is the restriction $p>D$. Of course, this
  has no influence on the asymptotics.
\item By far the most important reason that we can obtain irrationality
  results is that our recursion is of \emph{Ap\'ery type}, in other words
  that $d_D(n)$ involves only the \emph{lowest common multiple} and not the
  product of $Dj+1$. It can be shown that when, as in our family, $a(n)$ has
  degree $1$ and $b(n)$ has degree $2$, both with rational coefficients,
  this can happen if and only if either $a(n)=An+\cdots$, $b(n)=Bn^2+\cdots$ with $A^2+4B$ a nonzero rational square, or $a(n)=A(n-(x_1+x_2+1)/2)$ and
  $b(n)=B(n-x_1)(n-x_2)$ with $x_1$ and $x_2$ rational, which is the
  case for our family.
  \end{enumerate}
\end{remarks}

\subsection{Application to Irrationality Measures}

\smallskip

\begin{remarks}\begin{enumerate}
    \item It is immediate to check that the conditions of the theorem are satisfied for
      all of our examples.
    \item Since the contribution of $e_p(n)$ is at most logarithmic, it does not play
      any role in the logarithmic asymptotics of the denominators.
    \item By a numerical check, it seems that the above bound on the denominators of the
      rational approximations of all our CFs is asymptotically best possible.
  \end{enumerate}
\end{remarks}

\smallskip

The asymptotics of $d_D(n)$ and $d^*_D(n)$ are as follows:

\begin{proposition}\label{prop:dDstar}
  \begin{enumerate}\item[(1)] As $n\to\infty$ we have
    $$\log(d_D(n))\sim m_D\cdot n \quad\text{with}\quad m_D=\dfrac{D}{\phi(D)}\sum_{\substack{1\le j\le D\\\gcd(j,D)=1}}\dfrac{1}{j}\;.$$
    In particular, $m_2=2$, $m_3=9/4$, $m_4=8/3$, and $m_6=18/5$.
  \item[(2)] As $n\to\infty$ we have
    $$\log(d^*_D(n))\sim m^*_D\cdot n \quad\text{with}\quad
    m^*_D=m_D-(1/\phi(D))(\pi\cot(\pi/D)+D/(D-1)-D)\;.$$
    In particular $m_2^*=2$, $m_3^*=3-\pi/(2\sqrt{3})$, $m_4^*=4-\pi/2$, and
    $m_6^*=6-\pi\sqrt{3}/2$, which can also be written (only for these four values)
    $m_D^*=D-(\pi/2)\cot(\pi/D)$.
  \end{enumerate}
\end{proposition}

\begin{proof} Statement (1) is given in \cite{BKS02}. For (2), we use the following consequence of the
  prime number theorem (see \cite[Lemma 6]{Nes10} for a proof): for real $u<v$ from the interval
  $(0,1)$, as $n\to\infty$ we have $\sum_{p\text{ prime, }u\le\{n/p\}<v}\log p\sim\big(\psi(v)-\psi(u)\big)n$, where $\psi(x)=\Gamma'(x)/\Gamma(x)$ is the digamma function.
  Restricting the asymptotics to primes $p\le n$, that is, excluding the primes satisfying $u\le n/p<v$ from consideration, corresponds to the correction
\begin{equation}
\sum_{\substack{p\le n\\u\le\{n/p\}<v}}\log p\sim\Big(\psi(v)-\psi(u)+\frac1v-\frac1u\Big)n
\quad\text{as}\; n\to\infty\;.
\label{prime-asym}
\end{equation}
Furthermore, note that for any $C>1$, disregarding primes
$p\le \max(C\sqrt n,D)$ does not affect the asymptotics.

We use the asymptotics in \eqref{prime-asym} with $u=1/D$ and $v=(D-1)/D$, and apply the reflection formula for the $\psi$ function.
Furthermore, only primes $p\equiv-1\pmod D$ are taken into account, and the density of them among all primes $p\le n$ satisfying the fractional-part constraints is $1/\phi(D)$, proving the formula.
\end{proof}

Thanks to Proposition \ref{prop:asymp}, Theorem \ref{thm:denom},
and the above proposition, we can apply Lemma \ref{lem:irrmea} with
$F=\log(|E|)/2$ and $M=m_D^*$ to compute the irrationality measures.
We have thus proved the validity of the irrationality measures given in
Table \ref{tableall}, hence of Theorem \ref{thm:irrall}.

\subsection{More General Irrationality Measures}

Although our goal was to find irrationality measures for Chowla--Selberg
gamma quotients, which happen to be related to special values of quotients
of hypergeometric functions, it is clear that more generally our methods
apply to such quotients. A special case is as follows:

\begin{theorem}\label{thm:fD} As before, let $D=2,3,4$, or $6$ and let $f_D$
  be the function defined in Theorem \ref{thm:alt}. Let $z$ be a rational
  number such that $4D^3z\in\Z$ and $|z|>1$, and assume that $F(z)>m_D^*$ where
  $F(z)=\log(\sqrt{|z|}+\sqrt{|1-z|})$ and $m_D^*$ is given by Proposition
  \ref{prop:dDstar} (2).
  
  Then $f_D(1/z)$ is irrational, and its irrationality measure satisfies
  $$\mu(f_D(1/z))<1+\dfrac{F(z)+m_D^*}{F(z)-m_D^*}\;.$$
\end{theorem}

\begin{proof} Write $4D^2z=n/u$ with $n\in\Z$ and $u\mid D$ by assumption,
  and choose $A=\prod_{p\mid n}p^{\lceil v_p(n)/2\rceil}$, so that by
  definition we have $K=A^2/(4D^2z)=(A^2/n)u\in\Z$, and we have $A/K=(n/A)/u$,
  with $n/A\in\Z$. If $u=1$ the condition of Theorem \ref{thm:denom} (i) is
  satisfied. Otherwise, denoting by $g$ the denominator of $A/K$ we have
  $g\mid u\mid D$, and $v_p(K)\ge v_p(u)$, so
  the condition $v_2(K)\ge 2v_2(g)-2$ will be satisfied if $v_2(u)\le 2$,
  and $v_3(K)\ge 2v_3(u)-1$ if $v_3(u)\le 1$, which is the case since $u\mid D$
  and $D=2,3,4$, or $6$, so the condition of (ii) is satisfied and the
  conclusion follows.\end{proof}
  
Of course, the same result is true for the other representations of the
continued fraction limits, in particular for the function
$z^{1/2}(T_1/T_0)(1/2,1-1/D,1/2+1/z^{1/2})$ defined in Theorem \ref{thm:main2}
and for the complete elliptic integrals in the case $D=2$, for which the
result is the following:

\begin{corollary}
  Let $k$ be a complex number such that $128(1-k^2)/k^4\in\Z$ and $\Re(k^2)<1$,
  and assume that $F(k)>2$ where $F(k)=\log((|2-k^2|+2\sqrt{|1-k^2|})/|k|^2)$.

  Then $Q(k)=(k^2-2)K(k)/E(k)$ is irrational, and its irrationality
  measure satisfies
  $$\mu(Q(k))<\dfrac{F(k)+2}{F(k)-2}\;.$$
\end{corollary}

\begin{proof} Simply apply the previous theorem to $D=2$ and $z=(2-k^2)^2/k^4$.
  Note that in general we do \emph{not} have
  $|2-k^2|+2\sqrt{|1-k^2|}=(1+\sqrt{|1-k^2|})^2$.
\end{proof}

\section{Possible Generalizations}

There are also continued fractions attached to some other Chowla--Selberg
gamma quotients and corresponding to other values of $R_i(\tau)$
or $S_i(\tau)$. These do not possess any obvious arithmetic applications.

\smallskip

On the other hand, if we use the same type of CF as the rapidly convergent one
for $2^{1/3}$ which led to our motivating example, we can almost certainly
obtain irrationaliy measures for new constants, but the main problem is,
which constants? For instance, using good $abc$ triples analogue to
$2^7-5^3=3$ and Laguerre's CF \ref{prop:lag}, we can write very rapidly
convergent continued fractions for $3^{1/3}$ and $5^{1/4}$ and shift them as
we did for $2^{1/3}$. The CF for $5^{1/4}$ is given in Section
\ref{sec:unshift}, so we know that its shifted version gives the CF labeled
(2.3). The CF for $3^{1/3}$ is
$$3^{1/3}=[[13/9,6581,6576(2n-1)],[-130/9,-25(9n^2-1)]]\;,$$
and its shifted version is thus M\"obius-equivalent to
$$C=[[0,13162,26304(n-1)],[1,-25(6n-1)(6n-5)]\;,$$
which of course would fit nicely in our table of CFs, with $D=6$, $A=26304$,
and $K=25$. Is this new constant $C$ connected to any known constants ? We do
not know, but we suspect not. Note that since $z=A^2/(4D^2K)=4804864/25$,
it does not quite satisfy the condition $4D^3z\in\Z$ of Theorem \ref{thm:fD},
but an easy modification of our methods would nonetheless give an irrationality
measure for $C$.

\smallskip

More promising should be the use of \emph{cocompact} arithmetic
triangle groups $(p,q,r)$. Recall that if $(a,b,c)$ are the parameters of
a ${}_2F_1$ with $0<a\le b,c<1$, the corresponding triangle group is given
up to permutation of $(p,q,r)$ by $1/p=1-c$, $1/q=c-a-b$, and $1/r=b-a$.
It is immediate to check that the condition $c=2a$ or $c=2b$ imposed by
our construction is equivalent to two of $p$, $q$, and $r$ being equal.
We have already seen this above when $r=\infty$. But there are several
dozen other arithmetic triangle groups satisfying this condition,
and if we could find analogues of Theorem \ref{thm:modeval} which would
involve automorphic forms on Shimura curves, this may give us more examples.

\smallskip

It should be emphasized that the continued fractions that we use are
\emph{very simple}, with $a(n)$ of degree $1$ and $b(n)$ of degree $2$.
The continued fractions given by the authors in \cite{Coh-Zud} for
\emph{squares} of Chowla--Selberg gamma quotients lie deeper (in our opinion).
In Appendix II we give a list of such CFs, some of them being found
experimentally. The reader will notice that there seems to be some correlation
between these CFs and those found in Table~\ref{tableall} above: the ones
involving $\CS(-3)$ seem related to 2.5, 4.4, and 3.2, the ones involving
$\CS(-4)$ to 2.1, 2.1, 4.1, and 2.2, and the last three involving $\CS(-8)$ to
4.3, 4.3, and 1.2.

It would be interesting to know whether these CFs can be generalized.
In addition, we have mentioned that the 44 CM values in Table~\ref{tablerat}
are the same as those given in \cite{Coh-Gui} to obtain rational hypergeometric
formulas for $1/\pi$. Could the above CFs for squares of CS values be related
to the rational hypergeometric formulas for $1/\pi^2$, hence thanks to the
work of \cite{DPVZ}, to hypergeometric CM evaluations of \emph{Hilbert modular
forms}, this time linked to ${}_3F_2$ instead of ${}_2F_1$ ?

\smallskip

To the authors' knowledge, previous irrationality measures \emph{always}
came (or could be interpreted) from series and/or integrals. Even Ap\'ery's
proofs which were initially given using continued fractions, were
reinterpreted in terms of integrals by F.~Beukers, and then improved.
The case of Chowla--Selberg gamma quotients seems totally different:
we do not know of any \emph{usable} series or integrals giving these
quantities. There do exist many hypergeometric series representations of
$\CS(-3)$ and $\CS(-4)$ for instance, but we have not found a single one
which is not the value at $z=\pm1$ of a hypergeometric function, hence with
very slow convergence, although it is plausible that some sporadic such
examples exist. And we do not know of any interesting integral
representation, outside of the trivial integral representations of the
corresponding hypergeometric function at $z=\pm1$. This is in contrast with
other gamma quotients such as beta function values, for which both interesting
series and integrals exist, although not sufficiently good to
prove irrationality results.

\smallskip

Finally, note that the $p$-adic analogue of the Chowla--Selberg formula is the
Gross--Koblitz formula, so that we could hope for a parallel development
leading to proofs of the irrationality of the corresponding $p$-adic
quantities.

\smallskip

\section*{Appendix I}

For completeness, we list continued fractions not belonging to our family but
obtained using other modular parametrizations and/or Ap\'ery type techniques.
All such CFs belong to doubly or triply infinite families of the form
$${\cal C}=[[a_0,a_1,An+A_0],[b_0,-K(D_1n+d_1)(D_2n+d_2)]]$$
with fixed $(A,K,D_1,D_2)$, hence having the same exponent of convergence $E$,
and in the following table we give one example from each family together with
$E$.

We believe that none of these CFs lead to an irrationality proof.

\begin{table}[H]
  \caption{Other Continued Fractions}
  \label{tableaux}
\begin{center}
  %  \scriptsize
  \begin{tabular}{|c||c|c|c|c|c|c|c|c|c|}
    \hline
    $L$ & $a_0$ & $a_1$ & $A$ & $A_0$ & $b_0$ & $K$ & $D_1n+d_1$ & $D_2n+d_2$ & $E$ \\
    \hline\hline
    $\CS(-3)$ & $15$ & $21$ & $14$ & $7$ & $-96$ & $8$ & $2n+1$ & $3n+4$ & $4/3$ \\
    $\CS(-3)$ & $15$ & $16$ & $14$ & $-3$ & $-96$ & $8$ & $2n+1$ & $3n-1$ & $4/3$ \\
    $\CS(-3)$ & $15$ & $49$ & $14$ & $3$ & $-480$ & $-32$ & $3n+2$ & $6n+1$ & $-16/9$ \\
    $\CS(-3)$ & $15$ & $64$ & $14$ & $33$ & $-480$ & $-32$ & $3n-2$ & $6n-1$ & $-16/9$ \\
    $\CS(-3)$ & $9/2$ & $7$ & $4$ & $3$ & $36$ & $-2$ & $2n+3$ & $3n+4$ & $-3$ \\
    $\CS(-3)$ & $0$ & $2$ & $30$ & $-35$ & $9$ & $16$ & $3n-2$ & $3n-2$ & $4$ \\
    $\CS(-3)$ & $0$ & $1$ & $10$ & $-11$ & $3$ & $4$ & $2n-1$ & $2n-1$ & $4$ \\
    $\CS(-3)$ & $0$ & $1$ & $21$ & $-22$ & $18$ & $-8$ & $3n-1$ & $3n-1$ & $-8$ \\
    $\CS(-3)$ & $0$ & $2$ & $21$ & $-20$ & $18$ & $-8$ & $3n-2$ & $3n-2$ & $-8$ \\
    $2^{1/3}\CS(-3)$ & $0$ & $-1$ & $8$ & $-10$ & $12$ & $-3$ & $4n-1$ & $4n-1$ & $-3$ \\
    $2^{1/3}\CS(-3)$ & $0$ & $1$ & $8$ & $-6$ & $12$ & $-3$ & $4n-3$ & $4n-3$ & $-3$ \\
    $2^{1/3}(2+\sqrt{3})\CS(-3)$ & $30$ & $12$ & $7$ & $5$ & $90$ & $-1$ & $2n+3$ & $4n+5$ & $-8$ \\
    $2^{1/3}(2+\sqrt{3})\CS(-3)$ & $0$ & $1$ & $40$ & $-50$ & $24$ & $9$ & $4n-3$ & $4n-3$ & $9$ \\
    $2^{1/3}(2+\sqrt{3})\CS(-3)$ & $0$ & $1$ & $20$ & $-24$ & $12$ & $9$ & $2n-1$ & $2n-1$ & $9$ \\
    $\CS(-4)$ & $0$ & $3$ & $40$ & $-42$ & $24$ & $9$ & $4n-3$ & $4n-3$ & $9$ \\
    $\CS(-4)$ & $0$ & $5$ & $40$ & $-38$ & $24$ & $9$ & $4n-1$ & $4n-1$ & $9$ \\
    $\CS(-7)$ & $0$ & $13$ & $248$ & $-250$ & $168$ & $-63$ & $4n-1$ & $4n-1$ & $-63$ \\
    $\CS(-7)$ & $0$ & $15$ & $248$ & $-246$ & $168$ & $-63$ & $4n-3$ & $4n-3$ & $-63$ \\
    $(4+\sqrt{7})\CS(-7)$ & $0$ & $33$ & $520$ & $-650$ & $1008$ & $63^2$ & $4n-3$ & $4n-3$ & $81/49$ \\
    $(4+\sqrt{7})\CS(-7)$ & $0$ & $41$ & $260$ & $-276$ & $504$ & $63^2$ & $2n-1$ & $2n-1$ & $81/49$ \\
    \hline
  \end{tabular}
\end{center}
\end{table}

\medskip

\section*{Appendix II}

We give here a list of CFs involving higher powers of Chowla--Selberg gamma
quotients, some obtained experimentally (hence not all proved). 

For squares, these CFs are of the form
$$L=[[a_0,a_1,A_2n^2+A_0],[b_0,-Kf(D_1)f(D_2)]]\quad\text{with}\quad f(D)=(Dn+1)(Dn+D-1)\;,$$
and as usual we also give the exponential rate of convergence $E$.

\begin{table}[H]
  \caption{Continued Fractions for $\CS(D)^2$}
  \label{tableaux2}
\begin{center}
  %  \scriptsize
  \begin{tabular}{|c||c|c|c|c|c|c|c|c|c|}
    \hline
    $L$ & $a_0$ & $a_1$ & $A_2$ & $A_0$ & $b_0$ & $K$ & $D_1$ & $D_2$ & $E$ \\
    \hline\hline
    $(2^{1/3}\CS(-3))^2$ & $36$ & $21$ & $28$ & $2$ & $324$ & $3$ & $2$ & $4$ & $4/3$ \\
    $(2^{1/3}\CS(-3))^2$ & $180$ & $245$ & $108$ & $2$ & $-24300$ & $-9$ & $4$ & $6$ & $-4$ \\
    $(2^{1/3}\CS(-3))^2$ & $72$ & $33$ & $40$ & $2$ & $648$ & $9$ & $2$ & $2$ & $9$ \\
    $(2^{1/3}\CS(-3))^2$ & $90$ & $159$ & $164$ & $4$ & $810$ & $9$ & $2$ & $3$ & $81$ \\
    $\CS(-4)^2$ & $48$ & $15$ & $20$ & $1$ & $288$ & $1$ & $2$ & $4$ & $4$ \\
    $\CS(-4)^2$ & $96$ & $125$ & $96$ & $2$ & $-2592$ & $-3$ & $4$ & $4$ & $-(2+\sqrt{3})^2$ \\
    $\CS(-4)^2$ & $80$ & $47$ & $44$ & $1$ & $-160$ & $-1$ & $2$ & $2$ & $-((1+\sqrt{5})/2)^{10}$ \\
    $(12^{1/4}\CS(-4))^2$ & $288$ & $100$ & $90$ & $2$ & $-2304$ & $-2$ & $3$ & $3$ & $-(2+\sqrt{3})^3$ \\
    $\CS(-8)^2$ & $320/7$ & $1287/7$ & $236$ & $21$ & $163840/49$ & $384$ & $2$ & $3$ & $32/27$ \\
    $\CS(-8)^2$ & $64$ & $9$ & $12$ & $1$ & $256$ & $2$ & $2$ & $2$ & $2$ \\
    $\CS(-8)^2$ & $192$ & $45$ & $28$ & $1$ & $-3072$ & $-8$ & $2$ & $2$ & $-8$ \\
    $\CS(-8)^2$ & $320/3$ & $99$ & $116$ & $3$ & $6400/3$ & $6$ & $2$ & $6$ & $27/2$ \\
    \hline
  \end{tabular}
\end{center}
\end{table}

We have not been able to find CFs of this type involving $\CS(-7)^2$, and
as mentioned above, there seems to be some correlation between these CFs
and those of Table~\ref{tableall}.

\smallskip

For higher powers we found
\begin{align*}
  \CS(-3)^3&=[[216,24,24n^3+6n],[1296,-(2n+1)^4(3n+1)(3n+2)]]\quad (P=1/3) \\
  \CS(-4)^4&=[[2816,635,688n^4+200n^2+3],[720896,-192(2n+1)^6(3n+1)(3n+2)]]\quad (E=27/16) \\
  \CS(-4)^4&=[[5376,1065,1040n^4+120n^2+1],[516096,-16(2n+1)^6(4n+1)(4n+3)]]\quad (E=64)
\end{align*}
where the first does not converge exponentially fast, but on the contrary
very slowly as $\CS(-3)^3-p(n)/q(n)\sim C/n^{1/3}$ with
$C=9\cdot2^{-1/3}\CS(-3)^{5/3}$.

\medskip

\subsection*{Acknowledgments}
The work of the second author was supported in part from the NWO grant OCENW.M.24.112.
The authors thank Mourad Ismail and Walter Van Assche for related discussions on the associated ultraspherical polynomials, as well as Li Lai for critical comments.
We further thank Frank Calegari and the anonymous referees for valuable feedback.

\end{document}